\documentclass[10pt]{article}

\usepackage{amsmath,amsthm}
\usepackage{amssymb}
\usepackage{color}
\usepackage{breqn}
\usepackage{enumerate}
\usepackage{graphicx}
\usepackage{bbm}
\usepackage[affil-it]{authblk}
\usepackage{tabu}

\newtheorem{theorem}{Theorem}[section]
\newtheorem{corollary}[theorem]{Corollary}
\newtheorem{lemma}[theorem]{Lemma}
\newtheorem{proposition}[theorem]{Proposition}
\newtheorem{question}[theorem]{Question}

\newtheorem{conjecture}[theorem]{Conjecture}

\newtheorem{observation}[theorem]{Observation}
\theoremstyle{definition}

\makeatletter
\def\blfootnote{\gdef\@thefnmark{}\@footnotetext}
\makeatother	
	
\begin{document}

\date{}
\title{Small Sets with Large Difference Sets}
\author{Luka Mili\'{c}evi\'{c}
\thanks{E-mail address: \texttt{lm497@cam.ac.uk}}}
\affil{Department of Pure Mathematics and Mathematical Statistics\\Wilberforce
Road\\Cambridge CB3 0WB\\UK}

\maketitle

\blfootnote{\textup{2010} \textit{Mathematics Subject Classification}: 11B13; 11P99}

\abstract{For every $\epsilon > 0$ and $k \in \mathbb{N}$, Haight constructed a set $A \subset \mathbb{Z}_N$ ($\mathbb{Z}_N$ stands for the integers modulo $N$) for a suitable $N$, such that $A-A = \mathbb{Z}_N$ and $|kA| < \epsilon N$. Recently, Nathanson posed the problem of constructing sets $A \subset \mathbb{Z}_N$ for given polynomials $p$ and $q$, such that $p(A) = \mathbb{Z}_N$ and $|q(A)| < \epsilon N$, where $p(A)$ is the set $\{p(a_1, a_2, \dots, a_n)\colon a_1, a_2, \dots, a_n \in A\}$, when $p$ has $n$ variables. In this paper, we give a partial answer to Nathanson's question. For every $k \in \mathbb{N}$ and $\epsilon > 0$, we find a set $A \subset \mathbb{Z}_N$ for suitable $N$, such that $A- A = \mathbb{Z}_N$, but $|A^2 + kA| < \epsilon N$, where $A^2 + kA = \{a_1a_2 + b_1 + b_2 + \dots + b_k\hspace{2pt}\colon\vspace{2pt} a_1, a_2,b_1, \dots, b_k \in A\}$. We also extend this result to  construct, for every $k \in \mathbb{N}$ and $\epsilon > 0$, a set $A \subset \mathbb{Z}_N$ for suitable $N$, such that $A- A = \mathbb{Z}_N$, but $|3A^2 + kA| < \epsilon N$, where $3A^2 + kA = \{a_1a_2 + a_3a_4 + a_5a_6 + b_1 + b_2 + \dots + b_k\hspace{2pt}\colon\vspace{2pt} a_1, \dots, a_6,b_1, \dots, b_k \in A\}$.}

\section{Introduction}
The problem of comparing different expressions involving the same subset $A$ of an abelian group $G$ (e.g. $A+A$ and $A-A$) is one of the central topics in additive combinatorics. For example, one of the starting points in the study of this field is the Pl\"{u}nnecke-Ruzsa inequality that bounds $|kA - lA|$ in terms of $|A|$ and $|A+A|$.

\begin{theorem}(Pl\"{u}nnecke-Ruzsa inequality,~\cite{Plunnecke},~\cite{RuzsaIneq}) Let $A$ be a subset of an abelian group. Then, for any $k,l \geq 1$ we have
$$|kA - lA| |A|^{k+l - 1} \leq |A+A|^{k+l}.$$\end{theorem}

To illustrate the difficulties in determining the right bounds for such inequalities, we note that even for the comparison of $|A+A|$ and $|A-A|$ the right exponents are not known. In fact, the best known lower bounds for $|A+A|$ in terms of $|A-A|$ have not changed for more than 40 years.

\begin{theorem}[Freiman, Pigaev; Ruzsa, \cite{FreimanBounds}, \cite{RuzsaBounds}] Let $A$ be a subset of an abelian group. Then $|A-A|^{3/4} \leq |A+A|$. \end{theorem}

In the opposite direction, the best known lower bound is given by the following result.

\begin{theorem}[Hennecart, Robert, Yudin, \cite{LowerBounds}] There exist arbitrarily large sets $A \subset \mathbb{Z}$ such that $|A+A| \leq |A-A| ^ {\alpha+ o(1)}$, where $\alpha\colon= \log(2)/\log(1 +\sqrt{2}) \approx 0.7864$. \end{theorem}
 
In 1973, Haight~\cite{Haight} found for each $k$ and $\epsilon > 0$, an integer $q$ and a set $A \subset \mathbb{Z}_q$ such that $A-A = \mathbb{Z}_q$ and $|kA| \leq \epsilon q$. Recently, Ruzsa~\cite{Ruzsa1} gave a similar construction, and observed that Haight's work even gives a constant $\alpha_k >0 $ for each $k$ with the property that there are arbitrarily large $q$ with sets $A \subset \mathbb{Z}_q$ such that $A-A = \mathbb{Z}_q$ and $|kA| \leq q^{1-\alpha_k}$. The ideas in both constructions are relatively similar, but Ruzsa's argument is cosiderably more concise.\\

In~\cite{Nathanson}, Nathanson applied Ruzsa's method to construct sets $A\subset R$ with $A- A = R$, but $kA$ small, for rings $R$ that are more general than $\mathbb{Z}_q$. In the same paper, he posed the following more general question. Given a polynomial $F(x_1, x_2, \dots, x_n)$ with coefficients in $\mathbb{Z}$, and a set $A \subset \mathbb{Z}_N$, write $F(A) = \{F(a_1, a_2, \dots, a_n) \colon a_1, \dots, a_n \in A\}$. His question can be stated as: given two polynomials $F, G$ over $\mathbb{Z}$ and $\epsilon > 0$, does there exist arbitrarily large $N$ and a set $A \subset \mathbb{Z}_N$ such that $F(A) = \mathbb{Z}_N$, but $|G(A)| < \epsilon N$?\footnote{Actually, Nathanson poses this question for more general rings $R$, but for $R = \mathbb{Z}$, the formulation we give here is a natural one.}\\

Let us now state the main result of this paper, which answers the first interesting cases of Nathanson's question. Once again we recall the notation
\[A^2 + kA = \{a_1a_2 + a'_1 + a'_2 + \dots + a'_k\hspace{2pt}\colon\vspace{2pt} a_1, a_2,a'_1, \dots, a'_k \in A\},\]
and more generally,
\[lA^2 + kA = \{a_1a_2 + \dots + a_{2l-1}a_{2l} +  a'_1 + a'_2 + \dots + a'_k\hspace{2pt}\colon\vspace{2pt} a_1, a_2, \dots, a_{2l}, a'_1, \dots, a'_k \in A\}.\]

\begin{theorem} Given $k \in \mathbb{N}_0$ and any $\epsilon > 0$, there is a natural number $q$ and a set $A \subset \mathbb{Z}_q$ such that
\[A - A = \mathbb{Z}_q,\text{ but }|A^2 + kA| \leq \epsilon q.\]\end{theorem}

In fact we prove rather more.

\begin{theorem}\label{mainResult}For $l \in \{1,2,3\}$, any $k \in \mathbb{N}_0$ and any $\epsilon > 0$, there is a natural number $q$ and a set $A \subset \mathbb{Z}_q$ such that
\[A-A = \mathbb{Z}_q,\text{ but }|lA^2 + kA| < \epsilon q.\]
Moreover, we can take $q$ to be a product of distinct primes, and we can take the smallest prime dividing $q$ to be arbitrarily large.
\end{theorem}

We shall discuss each of the cases $l=1,2,3$ separately. Note also an interesting phenomenon in the opposite direction. Namely, if we are not allowed freedom in the choice of the modulus, a statement like the theorem above cannot hold. The reason is that, by the result of Glibichuk and Rudnev (Lemma 1 in ~\cite{Glib}) whenever $A \subset \mathbb{F}_p$ for a prime $p$, is a set of size at least $|A| > \sqrt{p}$, then $10A^2 = \mathbb{F}_p$ (and $A-A = \mathbb{F}_p$ certainly implies $|A| > \sqrt{p}$). Hence, unlike the linear case, already for quadratic expressions we have strong obstructions. 

In fact, this problem is comparable in spirit to sum-product phenomenon, which can be stated as the following theorem.

\begin{theorem} (Bourgain, Katz, Tao~\cite{BKT}, Sum-product estimate.) Let $\delta > 0$ be given. Then there is $\epsilon > 0$ such that whenever $A \subset \mathbb{Z}_q$ for a prime $q$ satisfies
\[q^{\delta} < |A| < q^{1-\delta},\]
then one has
\[\max\{|A^2|, |2A|\} \geq |A|^{1 + \epsilon}.\]
\end{theorem}

This was further generalized to arbitrary modulus $q$.

\begin{theorem}\label{bgsthm} (Bourgain~\cite{Bcomp}, Sum-product estimate for composite moduli.) Given $q, q'$ such that $q' | q$, write $\pi_{q'}$ for the natural projection from $\mathbb{Z}_q \to \mathbb{Z}_{q'}$.\\
Let $\delta > 0$ be given. We then have $\epsilon, \eta > 0$ such that the following holds. Whenever $A \subset \mathbb{Z}_q$ satisfies
\[|A| \leq q^{1-\delta}\]
and,
\[|\pi_{q'}(A)| \geq {q'}^{\delta}\text{ for all }q'|q,\text{ with }q' \geq q^{\eta},\]
then
\[\max\{|A^2|, |2A|\} \geq |A|^{1 + \epsilon}.\] 
\end{theorem}

Hence, the sum-product phenomenon still holds in general $\mathbb{Z}_N$, even when N is composite, and given the similarity of our problem, it could well be that the result of Glibichuk and Rudnev stated above holds in the more general setting as well. (Note that if $A-A = \mathbb{Z}_q$, then it satisfies the technical condition in  Theorem~\ref{bgsthm}.) 

\begin{conjecture} There is $l$ such that whenever $A \subset \mathbb{Z}_q$ and $A-A = \mathbb{Z}_q$, then we have $lA^2 + lA = \mathbb{Z}_q$.\end{conjecture}

\subsection{Acknowledgements}
I would like to thank Trinity College Cambridge and the Department of Pure Mathematics and Mathematical Statistics of Cambridge University for their generous support, and Imre Leader for encouragement and helpful discussions concerning this paper. 

\section{Overview of the Construction}

\indent We begin the paper by reviewing Ruzsa's construction and generalizing its main ideas slightly to the context of polynomial expressions in $A$. As it turns out, to be able to construct a set $A$ such that $A-A = \mathbb{Z}_q$, but $|lA^2 + kA| = o(q)$, it will suffice to consider expressions which are sums of terms of the form $\alpha_i(x_i) + cx_i$, $(\alpha_i(x_i) + cx_i)(\alpha_i(x_i) + c'x_i)$ and $(\alpha_i(x_i) + cx_i)(\alpha_j(x_j) + c'x_j)$, with $c,c' \in \{0,1\}$ and then choose the maps so that the number of values attained by each expression is small. For example, one of the expressions that we have to consider already for the case $l=1$ is $\alpha_1(x_1) \alpha_2(x_2) + \alpha_1(x_1) + x_1 + \alpha(x_3)$. This discussion takes place in Section 3 and the rest of the paper is devoted to constructions of maps for various expressions.\\
 
In Section 4, we construct sets $A$ such that $A-A = \mathbb{Z}_q$ but $A^2 + kA$ is small. In this construction, we come to a basic version of one of the main ideas, which we call the \emph{identification of coordinates}. Very roughly, if $q$ is a product of distinct prime $p_1p_2 \dots p_n$, using approximate homomorphisms between $\mathbb{Z}_{p_i}$ and $\mathbb{Z}_{p_j}$, we can essentially treat $\mathbb{Z}_q$ as a vector space of dimension $n$. Then, altough we might not ensure that each coordinate attains few values, we can ensure that their sum attains few values.\\

In Section 5, we construct sets $A$ such that $A-A = \mathbb{Z}_q$ but $2A^2 + kA$ is small. There, we improve our results for the expression that involve a single variable using a variant of Weyl's equidistribution theorem for polynomials. Using this result, the identification of coordinates is developped further and we conclude this section with the strongest form of identification of coordinates.\\

The final part of the construction, finding sets $A$ with $3A^2 + kA$ small, is carried out in Section 6. There, we also touch upon some limitations of the usual approach and therefore develop different ideas to treat some of the remaining expressions. Namely, for certain choices of coefficients, in the expression
$$(\alpha(x) + c_1 x)(\beta(y) + c_2y) + (\alpha(x) + c_3 x)(\beta(y) + c_4y) + (\alpha(x) + c_5x)(\gamma(z) + c_6z)$$
the identification of coordinates cannot work. For this expression, we give a different, probabilistic argument.\\

The final section is devoted to some open problems and questions that naturally arise, including the motivation for some of these. We have tried to organize the paper so that methods used naturally develop from the case $A^2 + kA$ to the case $3A^2 + kA$, highlighting the new difficulties that arise and why the earlier arguments  are not powerful enough for the later expressions.

\section{Overview of Ruzsa's argument and Initial Steps} We now briefly discuss Ruzsa's construction of sets $A \subset \mathbb{Z}_q$ such that $A-A = \mathbb{Z}_q$, but $|kA| = o(q)$. His ideas will be important for the later constructions given in this paper.\\
Let us first analyse the requirement that $A-A = \mathbb{Z}_q$. Given any $x \in \mathbb{Z}_q$, we thus have $y \in A$ such that $y + x \in A$. If we write $\varphi(x)$ for such a $y$, this yields a map $\varphi: \mathbb{Z}_q \to \mathbb{Z}_q$ with the property that all $\varphi(x)$ and $\varphi(x) + x$ are contained in $A$. Removing all other elements from $A$ does not change the equality $A-A = \mathbb{Z}_q$, and it can only make $kA$ smaller, so Ruzsa's starting point is to consider a set $A$ of the form
$$\{\varphi(x): x \in \mathbb{Z}_q\} \cup \{\varphi(x) + x: x \in \mathbb{Z}_q\},$$
where $\varphi$ is map from $\mathbb{Z}_q$ to itself. We shall do the same in this paper as well, and throughout the paper we will devote ourselves to finding suitable modulus $q$ and maps on $\mathbb{Z}_q$.\\

Thus, we have to understand how to find a suitable $q$ and a map $\varphi$ which then give rise to the desired set $A$. Let us now examine the elements of $kA$. These are sums $a_1 + a_2 + \dots + a_k$, where $a_i \in A$. But each element of $A$ is either $\varphi(x)$ or $\varphi(x) + x$ for some $x \in \mathbb{Z}_q$. Hence, elements of $kA$ are of the form
$$\sum_{i \in I} \varphi(x_i) + \sum_{i \notin I} (\varphi(x_i) + x_i)$$
for a subset $I \subset[k]$ and $x_1, x_2, \dots, x_k$. Immediately we see that the number of different expressions here is bounded in terms of $k$ (in fact, it equals $2^k$). Further, we consider which of the $x_i$ are equal, grouping the corresponding terms $\varphi(x_i)$ and $\varphi(x_i) + x_i$ together, and renaming the variables along the path to $y_1, y_2, \dots, y_s$. Hence, every element of $kA$ is of the form
\begin{dmath}\label{yexpressions}\sum_{i = 1}^s (a_i \varphi(y_i) + b_i y_i),\end{dmath}
where $s \leq k$, $k \geq a_i \geq b_i \geq 0$ and all $y_1, \dots, y_s$ are different. Once again, treating $y_i$ as formal variables, the number of expressions we wrote is bounded in terms of $k$. The plan now is to make sure that each such expression attains a small number of values, so that in total only at most $\epsilon q$ values attained.\\

Ruzsa's main idea in the costruction is the \emph{separation of functions}, which we now discuss. In all these expressions we have the same map $\varphi$ occuring. However, we can turn the problem of constructing a single function $\varphi$ that works for all expressions into a much easier problem of constructing a function for each expression separately. We first list all the expressions of the form~(\ref{yexpressions}), sorted in the asscending order by the number of variables appearing. Thus, our list start from expressions of the form $a \varphi(y) + b$. Next, we split $q$ as a product of coprime numbers $q = q_1 q_2 \dots q_r$, with one $q_i$ for each expression so that by Chinese Remainder Theorem we have $\mathbb{Z}_q = \mathbb{Z}_{q_1} \oplus \mathbb{Z}_{q_2} \oplus \dots \oplus \mathbb{Z}_{q_r}$.

We promise that however we choose an expression and values of $y_i$, we get at least one zero coordinate (which need not depend on the expression) and we call this \textbf{ZCP} (Zero Coordinate Promise). If $i$-th expression has only one variable appearing, thus it is of the form $a\varphi(y) + b y$, we can easily ensure \textbf{ZCP} by setting the $i$-th component of the function as $\varphi_i (y) = - b a^{-1} y_i$. Now, take any expression    
$$\sum_{i = 1}^s (a_i \varphi(y_i) + b_i y_i),$$
and assume that for every such expression with fewer than $s$ variables \textbf{ZCP} holds. Let $q'$ be the product of $q_i$ for the expressions with fewer than $s$ variables. Note that, if we are given $y_1, y_2, \dots, y_s$, and if any two among them have the same value in $\mathbb{Z}_{q'}$, by induction hypothesis, \textbf{ZCP} already holds. Hence, we may assume not only that $y_1, y_2, \dots, y_s$ are different, but that they are different modulo $q'$. Write $y'_i$ for the residue of $y_i$ mod $q'$. Then, looking at $j^{\text{th}}$ coordinate, we have to define $\varphi_j$ such that 
$$\sum_{i = 1}^s (a_i \varphi_j(y'_i, (y_i)_j) + b_i (y_i)_j)$$
equals zero for all choices of $y_1, \dots, y_s$ such that $y'_i$ are different. But, we can rewriting $\varphi_j(y'_i, (y_i)_j)$ as $\varphi_{j, y'_i}((y_i)_j)$ already tells us that we are actually looking for a new function for each variable! Hence, our goal is to find $s$ functions $\varphi_{j, y'_1}, \dots \varphi_{j, y'_s}$ such that the expression is once again zero. But linear maps once again work.

We start our own work in this paper by slightly generalizing Ruzsa's idea to polynomial setting. In what follows, by an \emph{$i$-degree term} we think of a product of $i$ terms of the from $\alpha_j(x_j)$ or $(\alpha_j(x_j) + x_j)$, the only rule being that indices of the map and variable to which it is applied (and which is possibly added) coincide. For example, $(\alpha_1(x_1) + x_1) \alpha_2(x_2) ^ 2$ and $\alpha_1(x_1) (\alpha_2(x_2) + x_2) (\alpha_3 (x_3) + x_3)$ are both $3$-degree terms, but $\alpha_1(x_2) \alpha_2(x_3) \alpha_3(x_1)$ is not, since the indices are not valid.

\begin{proposition}\label{RuzsaArg} Let $k$ be given, and let $a_1, a_2, \dots, a_k \in \mathbb{N}$. Suppose that for every $\epsilon > 0$ and every formal expression $E$ in functions $\alpha_i$ and variables $x_i$ of the form
$$\text{sum of }a_k \text{of } k\text{-degree terms} + \text{sum of }a_{k-1} \text{of } (k-1)\text{-degree terms}  + \dots + \text{sum of }a_1 \text{of } 1\text{-degree terms} ,$$ 
we can find a modulus $q$, which is a product of arbitrarily large distinct primes, and functions $\theta_i: \mathbb{Z}_q \to \mathbb{Z}_q$, so that the $E$ takes at most $\epsilon q$ values in $\mathbb{Z}_q$, when the functions $\theta_i$ are substituted in $E$. Then, for every $\epsilon > 0$, there is a modulus $Q$, product of arbitrarily large distinct primes, and a set $A \subset \mathbb{Z}_Q$ such that $A-A = \mathbb{Z}_Q$ and 
$$|a_k A^k + a_{k-1} A^{k-1} + \dots + a_1 A| \leq \epsilon Q.$$\end{proposition}

\begin{proof} We proceed as in the Ruzsa's construction (except that we do not insist of only having zero value in a coordinate, small number of values suffices). 
As before, we sort the expressions by the number of variables appearing, and process them in groups of those having the same number of a variables. We now turn to details.\\

Let $N = a_1 + a_2 + \dots + a_k$. Let $E_1, E_2, \dots, E_r$ be all the expressions in variables $y_1, y_2, \dots, y_N$ of the following form. Each expression is a sum of $a_k$ terms, each being a product of $k$ short terms $\varphi(y_i)$ or $\varphi(y_i) + y_i$, followed by $a_{k-1}$ terms which are products of $k-1$ short terms, etc. with a final contribution of $a_1$ terms, each being $\varphi(y_i)$ or $\varphi(y_i) + y_i$. As in the discussion before, these are all expressions that naturally arise from $a_k A^k + \dots + a_1 A$, when $A$ is defined as $\{\varphi(x)\colon x \in \mathbb{Z}_Q\} \cup \{\varphi(x) + x\colon x \in \mathbb{Z}_Q\}$. Comparing these expressions with the expressions in the assumptions of this proposition, we have that here only a single formal function appears, while in the other expressions we have a separate function for each variable. Let $m_0 = 0, m_1, m_2, \dots, m_N = r$ be indices such that if $m_i < j \leq m_{i+1}$, then the number of different variables among $(y_t)_{t=1}^N$ appearing in the expression $E_j$ is exactly $i+1$.\\  

Fix an increasing sequence $0 < \epsilon_1 < \epsilon_2 < \dots < \epsilon_N = \epsilon$. We inductively construct moduli $Q_1, Q_2, \dots, Q_N$ and functions $\varphi_i: Q_i \to Q_i$ such that for every $i \leq N$ we have that union of all images of expressions $E_1, E_2, \dots, E_{m_i}$ (that is all expressions futuring at most $i$ variables) takes at most $\epsilon_i Q_i$ values (when $\varphi_i$ is substituted in the expressions).\\

\indent\textbf{Base case: $i = 1$.} By the assumption, for every expression $E_i$ that has only one variable, we have moduli $q_i$ with arbitrarily large distinct prime factors, and a map $\theta^{(1)}_i$, such that $E_i$ takes only at most $\epsilon_1  q_i/ {m_1}$ values. Thus, w.l.o.g. $q_1, q_2, \dots, q_{m_1}$ are all coprime, with distinct arbitrarily large prime factors. We set $Q_1 = q_1 q_2 \dots q_{m_1}$ and identify $\mathbb{Z}_{Q_1}$ with $\mathbb{Z}_{q_1} \oplus \mathbb{Z}_{q_2} \oplus \dots \oplus \mathbb{Z}_{q_{m_1}}$, and we define $\varphi_1$ coordinate-wise as $\varphi_{1, i} (x) := \theta^{(1)}_i(x_i)$, where $x_i$ is $i$-th coordinate of $x$. Note that union of all values attained by these $m_1$ expressions with this definition of $Q_1$ and $\varphi_1$ has size bounded by 
$$\sum_{i = 1}^{m_1} |\operatorname{Im} E_i| \leq \sum _{i=1}^{m_1} \frac{\epsilon_1 q_i}{m_1} \frac{Q_1}{q_i} = \epsilon_1 Q_1,$$
as desired. (Here we write $\operatorname{Im} E_i$ for the resulting image of the expression $E_i$, and we have a trivial bound for it -- the expression may only take at most $\epsilon_1q_i/m_1$ values on the $i^{\text{th}}$ coordinate.)\\

\textbf{Inductive step.} Suppose now that we have found $\varphi_s: \mathbb{Z}_{Q_s} \to \mathbb{Z}_{Q_s}$  such that in total all expressions with at most $s$ variables have a small image $V_s$, i.e. only at most $\epsilon_s Q_s$ values are attained. We shall construct $Q_{s+1}$ as a product $Q_s R_{m_s + 1} R_{m_s + 2} \dots R_{m_{s+1}}$, where $R_i$ is an auxiliary modulus for the expression $E_i$, with the property that either $E_i$ takes one of the small number of values on $\mathbb{Z}_{Q_s}$ or a value in another small set in $\mathbb{Z}_{R_i}$. Here we use Ruzsa's \emph{separation of functions} idea.\\
Fix an expression $E_i$ with exactly $s + 1$ variables. If we take values of these variables restricted to $\mathbb{Z}_{Q_s}$, and it happens so that at least two such values coincide, then using the map $\varphi_s$ the value of the expression $E_i$ (also restricted to $\mathbb{Z}_{Q_s}$) is actually a value of one of the expressions we already considered, with at most $s$ variables, so it lies in the small set $V_s$. Hence, we only need to consider the choices of $y_1, y_2, \dots, y_{s+1}$ (w.l.o.g. these are the variables that appear) which differ in $\mathbb{Z}_{Q_s}$. We split the expression $E_i$ further into cases on $y_i$ mod $Q_s$, thus into further $L \leq Q_s^{s+1}$ cases. Pick an arbitrary choice $C$ of $s+1$ distinct values in $Q_s$. Look back at $E_i$ and change every appearance of $\varphi(y_t)$ by $\alpha_t(y_t)$. By assumptions, we have a choice of an integer $r_C$ with arbitrarily large distinct prime factors and maps $\theta^{(C)}_t$ such that the modified $E_i$ takes only at most $(\epsilon_{s+1} - \epsilon_s)r_C / ((m_{s+1} - m_s) Q_s^{s+1})$ values in $\mathbb{Z}_{r_C}$. Finally, define $R_i$ as the product of all these $r_C$, and $(\varphi_{s+1})_i(x)$ as follows: for every $C$, take $(\varphi_{s+1})_i(x)$ at the coordinate corresponding to $r_C$ to be zero if $x$ modulo $\mathbb{Z}_{Q_s}$ is not in $C$, otherwise, if it is the $j$-th residue, set $(\varphi_{s+1})_i(x) := \theta^{(C)}_j(x')$, where $x'$ is the coordinate of $x$ corresponding to $r_C$. It remains to check the size of images.\\
For every expression and every choice of values of $y_1, y_2, \dots, y_N$, we either end up in $A_s \times \mathbb{Z}_{R_{m_s + 1}} \times \mathbb{Z}_{R_{m_s + 2}} \times \dots \times \mathbb{Z}_{R_{m_{s + 1}}}$, which has size at most $\epsilon_s Q_{s+1}$, or one of the coordinates is in a fixed subset of $\mathbb{Z}_{R_t}$ of size at most $(\epsilon_{s+1} - \epsilon_s)R_t/(m_{s+1} - m_s)$. Summing everything together, the image has at most $\epsilon_{s+1} Q_{s+1}$ values as desired.\end{proof} 

The rest of the paper is therefore devoted to finding moduli $q$ and maps $\alpha_i\colon \mathbb{Z}_q \to \mathbb{Z}_q$ under which the expressions like $(\alpha_1(x_1) + x_1) (\alpha_2(x_2) + x_2) + \alpha_3(x_3)^2$ do not take too many values. Along the way, we also discuss related problems and questions.\\

\noindent \textbf{Notation.} Throughout the paper, greek letters $\alpha, \beta$ and $\gamma$ will be used for the maps appearing in the expressions. The following functions will be frequently used in our construction. For a prime $p$, we use the standard projection homomorphism $\pi_p \colon \mathbb{Z} \to \mathbb{Z}_p$, which sends integer $x$ to $x + p\mathbb{Z}$. Next, we define $\iota_p\colon\mathbb{Z}_{p} \to \mathbb{Z}$ by sending $x \in \mathbb{Z}_p$ to the integer $\iota_p(x) \in \{0, 1, \dots, p-1\} \subset \mathbb{Z}$ such that $\pi_p \circ\iota_p(x) = x$. For two primes $p$ and $q$, we also define the map $\operatorname{mod}_{p,q}\colon \mathbb{Z}_{p}\to \mathbb{Z}_{q}$ given by $\operatorname{mod}_{p,q} = \pi_q \circ \iota_p$. Finally, in any abelian group $Z$, and functions $f, g\colon S \to G$, from a set $S$ to $Z$, we write $f \overset{M}{=} g$ to mean that $\{f(s)-g(s)\colon s \in S\}$ is a set of size at most $M$. In particular, $f \overset{O(1)}{=} g$ means that $\{f(s)-g(s)\colon s \in S\}$ has a bounded size as $S$ grows. 

\section{Sets $A$ with small $A^2 + kA$}

The main result of this section is the case $l=1$ of the Theorem~\ref{mainResult}.

\begin{theorem}\label{1square} For any $k \in \mathbb{N}_0$ and any $\epsilon > 0$, there is a natural number $q$, which is a product of distinct, arbitrarily large primes, and a set $A \subset \mathbb{Z}_q$ such that $A-A = \mathbb{Z}_q$, while $|A^2 + kA| < \epsilon q$. \end{theorem}

\begin{proof} We start from the Proposition~\ref{RuzsaArg}. To be able to construct $A \subset \mathbb{Z}_q$ with full difference set, but small $A^2 + kA$, we need to handle the expressions that are sums of the \emph{quadratic part} which is a product of two terms of the form $\alpha_i(x_i) + x_i$ or $\alpha_i(x_i)$, and a \emph{linear part} which is itself a sum of $k$ summands, each being of the form $\alpha_i(x_i) + x_i$ or $\alpha_i(x_i)$. Note that for the terms in the linear part whose variables do not appear in the quadratic part, we can define the corresponding maps $\alpha_i$ to be affine so that the variables involved cancel out. Therefore, w.l.o.g. we only consider expressions whose variables appear already in the quadratic part. Note also that for the quadratic part we have two cases: either only one variable, w.l.o.g. $x_1$, appears, or exactly two variables, w.l.o.g. $x_1$ and $x_2$, appear. We treat these cases separately.\\

\noindent\textbf{Case 1: only one variable in the quadratic part.} Thus, our goal now is to show that if we are given a quadratic expression featuring only one variable, we can find a modulus and function, so that the expression takes a small number of values. In fact, here we do more and prove the claim for expressions of arbitrary degree.

\begin{lemma}\label{singleVarLemma}Let $d \in \mathbb{N}$ be given, and let $p > d$ be a prime. Then, given any maps $c_0, c_1, \dots, c_d\colon \mathbb{Z}_p \to \mathbb{Z}_p$ and any set $F \subset \mathbb{Z}_p$ of size less than $p/d$, we can find another map $\alpha\colon\mathbb{Z}_p \to \mathbb{Z}_p$ such that the expression
$$c_d(x) \alpha(x)^d + \dots c_1(x) \alpha(x) + c_0(x)$$
does not take a value in $F$ for any $x$ that has at least one of $c_1(x), c_2(x), \dots, c_d(x)$ non-zero.
\end{lemma}

\begin{proof} Suppose that for some $x$, we have that for every choice of $v = \alpha(x)$ we have $c_d(x) v^d + \dots c_1(x) v + c_0(x) \in F$. By the pigeonhole principle, some value $f \in F$ is hit at least $d+1$ times. Thus, the polynomial
$$c_d(x) v^d + \dots c_1(x) v + c_0(x) - f$$
has at least $d+1$ zeros, making it a zero polynomial. Hence $c_1(x), c_2(x), \dots, c_d(x)$ are simultaneously zero, proving the lemma.\end{proof}

\begin{corollary}\label{singleVarExpression} Let $E$ be an arbitrary $\mathbb{Z}$-linear combination of terms of the form $\alpha(x)^i x^j$, where at least one of such terms with $i > 0$ appears. Given any $\epsilon > 0$, we can find a modulus $q$, which is a product of distinct arbirtrarily large primes, and a map $\alpha\colon \mathbb{Z}_q \to \mathbb{Z}_q$ such that under $\alpha$ the expression $E$ takes at most $\epsilon q$ values in $\mathbb{Z}_q$.\end{corollary}
\begin{proof} Rewrite $E$ by grouping together a $\mathbb{Z}$-linear combination of $x^j$ that appear next to each $\alpha(x)^i$. Thus, we can write $E$ as $\alpha(x)^d f_d(x) + \dots + \alpha(x) f_1(x) + f_0(x)$, where each $f_i(x)$ is a polynomial in $x$ over $\mathbb{Z}$, and at least one of $f_1, f_2, \dots, f_d$ is not a zero polynomial. Let $D = \max \deg f_i$. Pick distinct arbitrarily large primes $p_1, p_2, \dots, p_t$, all w.l.o.g. larger than $2 d (D+1)$ and absolute values of coefficients of $f_1, f_2, \dots, f_d$ (so that non-zero polynomials do not become zero modulo $p_i$). By the Lemma~\ref{singleVarLemma}, we may find a map $\alpha_i\colon\mathbb{Z}_{p_i} \to \mathbb{Z}_{p_i}$ for each $i$ such that the image of $E$ has size at most $(1-1/d)p_i + 1$, when the variable $x$ ranges over values such that polynomials $f_1, f_2, \dots, f_d$ are not simlutaneously zero. But there are at most $D$ values of $x$ such that $f_1(x) = \dots = f_d(x) = 0$, so we conclude that modulo each $p_i$, the expression $E$ may take at most $(1-1/d)p_i + D + 1 \leq (1-1/2d) p_i$ values. Finally, set $q = p_1 p_2 \dots p_t$ and take $\alpha\colon \mathbb{Z}_q \to \mathbb{Z}_q$ to be $\alpha = (\alpha_1, \alpha_2, \dots, \alpha_t)$, where we as usual identify $\mathbb{Z}_q$ with $\mathbb{Z}_{p_1} \oplus \mathbb{Z}_{p_2} \oplus \dots \oplus \mathbb{Z}_{p_t}$. Hence, modulo $q$, the expression takes at most $(1-1/2d)^t q$ values. Taking $t$ large enough so that $(1-1/2d)^t < \epsilon$ proves the corollary.\end{proof}

The case 1 now follows by applying Corollary~\ref{singleVarExpression}.\\

\noindent \textbf{Case 2: the quadratic part has two variables.} The quadratic part must look like a product of two terms, each being either $\alpha_i(x_i) + x_i$ or $\alpha_i(x_i)$. By suitably renaming the variables, and adding $x_i$ to $\alpha_i(x_i)$ if necessary, w.l.o.g. we only need to consider the case when the quadratic part is $\alpha_1(x_1) \alpha_2(x_2)$, and the whole expression is
$$\alpha_1(x_1) \alpha_2(x_2) + L_1(x_1) + L_2(x_2)$$
where each $L_i(x_i)$ is a $\mathbb{Z}$-linear combination of $\alpha_i(x_i)$ and $x_i$. Note also that if $L_i(x_i)$ is nonzero, then $\alpha_i(x_i)$ appears with a nonzero coefficient.\\

We have now come to an important point in this paper, and one of the key ideas, which we shall now explain. We have to construct $q$ and maps $\alpha_1, \alpha_2\colon\mathbb{Z}_q \to \mathbb{Z}_q$ such that $\alpha_1(x_1) \alpha_2(x_2) + L_1(x_1) + L_2(x_2)$ takes $o(q)$ values. Suppose for a moment that the linear terms $L_i$ are both zero. Then, we have an easy way to make $\alpha_1(x_1) \alpha_2(x_2)$ constant, by setting one of the $\alpha_i$ to be zero. However, such an approach cannot work in the case when $L_1, L_2$ are not zero, as it would force one of the $L_i$ to be an affine map, which is surjective. As a way to overcome this, we can use both $\alpha_1 = 0$ and $\alpha_2 = 0$ to get additional freedom. Thus, we set $q = q_1 q_2$, where $q_1, q_2$ are coprime products of distinct primes, identify $\mathbb{Z}_q$ with $\mathbb{Z}_{q_1} \oplus \mathbb{Z}_{q_2}$, and set $\alpha_1$ to be zero on the first coordinate, and $\alpha_2$ to be zero on the second coordinate. Hence if $L_1(x_1) = \lambda_1 \alpha_1(x_1) + \mu_1 x_1$ and $L_2(x_2) = \lambda_2 \alpha_2(x_2) + \mu_2 x_2$, then the expression becomes
\begin{dmath}\label{twoCoordsExpression}\left(\mu_1 (x_1)_1 + \lambda_2 (\alpha_2)_1(x_2) + \mu_2 (x_2)_1, \lambda_1 (\alpha_1)_2(x_1) + \mu_1 (x_1)_2 + \mu_2 (x_2)_2\right).\end{dmath}
We now want to find $(\alpha_1)_2$ and $(\alpha_2)_1$ so that the expression~(\ref{twoCoordsExpression}) does not take too many values in $\mathbb{Z}_{q_1} \oplus \mathbb{Z}_{q_2}$. Suppose for a moment that instead of coprime $q_1$ and $q_2$ we actually had $q_1 = q_2$. Then, we could have simply taken
$$(\alpha_1)_2(x_1) := - \lambda_1^{-1} \mu_1 ((x_1)_1 + (x_1)_2)$$
and 
$$(\alpha_2)_1(x_2) := - \lambda_2^{-1} \mu_2 ((x_2)_1 + (x_2)_2),$$
which ensures that every value taken by the expression is of the form $(v, -v)$ and hence it is in small subset $\{(x,y): x+y = 0\}$ of $\mathbb{Z}_{q_1} \oplus \mathbb{Z}_{q_1}$. It turns out that we can use the same approach even if $q_1 \not= q_2$. We shall refer to this idea as the \emph{identification of coordinates}, which will appear at other places in this paper as well. The following proposition and its proof formalize this discussion. We slightly change the notation to make the reading easier.

\begin{proposition} (Basic identification of coordinates.) Let $\lambda_0, \lambda_1, \lambda_2, \mu_1, \mu_2 \in \mathbb{Z}$ be given and let $p \leq q$ be primes greater than $|\lambda_1|, |\lambda_2|$. Suppose that if $\lambda_1 = 0$ then $\mu_1 = 0$ and if $\lambda_2 = 0$ then $\mu_2 = 0$. Then we have $\alpha, \beta\colon \mathbb{Z}_p \oplus \mathbb{Z}_q \to \mathbb{Z}_p \oplus \mathbb{Z}_q$ such that
$$f\colon (x,y) \mapsto \lambda_0 \alpha(x) \beta(y) + \lambda_1 \alpha(x) + \mu_1 x + \lambda_2 \beta(y) + \mu_2 y$$
takes at most $O(q)$ values, when $x, y$ range over all pairs of values in $\mathbb{Z}_p \oplus \mathbb{Z}_q$.\label{basicdiffmoduliadd}\end{proposition} 

Recall the definition of map $\iota_p$ as the natural embedding of $\mathbb{Z}_p$ into $\mathbb{Z}$, the natural projection $\pi_p \colon \mathbb{Z} \to \mathbb{Z}_p$, and finally, the composition $\operatorname{mod}_{p,q}\colon \mathbb{Z}_p \to \mathbb{Z}_q$, given by $\operatorname{mod}_{p,q} = \pi_q \circ \iota_p$. Before proceeding with the proof, it is useful to note some easy properties of the maps $\iota_p$ and $\operatorname{mod}_{p,q}$.

\begin{lemma} \label{modProps} Let $p, p', p_1, p_2, p_3$ be primes. Then
\begin{itemize}
\item[](1) Given $z \in \mathbb{Z}$, we have $p | \iota_p(\pi_p(z)) - z$. Also, $\iota_p(\pi_p(z)) \leq z$, when $z \geq 0$.
\item[](2) Given $x, y\in \mathbb{Z}_{p}$, we have $\iota_p(x) + \iota_p(y) - \iota_p(x+y) \in \{0, p\}.$
\item[](3) Given $x, y\in \mathbb{Z}_{p}$, we have 
$$\operatorname{mod}_{p,p'}(x) + \operatorname{mod}_{p,p'}(y) - \operatorname{mod}_{p,p'}(x+y) \in \{0, \pi_{p'}(p)\} \subset \mathbb{Z}_{p'}.$$
\item[](4) Provided that $p_3 < (t+1) p_2$, we have
$$\operatorname{mod}_{p_2, p_1} \circ \operatorname{mod}_{p_3, p_2}(x) -\operatorname{mod}_{p_3, p_1}(x) \in\{-t\pi_{p_1}(p_2), -(t-1)\pi_{p_1}(p_2), \dots, 0\} \subset \mathbb{Z}_{p_1}.$$
\end{itemize}\end{lemma}

\begin{proof} \textbf{(1)} Applying $\pi_p$, we have $\pi_p(\iota_p(\pi_p(z)) - z) = \pi_p \circ \iota_p(\pi_p(z)) - \pi_p(z) = 0$, thus $p | \iota_p(\pi_p(z)) - z$. If $z \geq 0$, then $\iota_p(\pi_p(z)) - z \leq p-1$, so the claim follows.\\

\textbf{(2)} Let $x' = \iota_p(x), y' = \iota_p(y) \in \mathbb{Z}$. Note that $\pi_p(x' + y') = x+y$ and $x' + y' \in \{0, 1, \dots, 2p-2\}$. From definition, $\pi_p(\iota_p(x+y)) = x+y$ and $\iota_p(x+y) \in \{0, 1, \dots, p-1\}$. Hence, if we set $v = \iota_p(x) + \iota_p(y) - \iota_p(x+y)$, we have $p | v$ and $v \in \{-(p-1), -(p-2), \dots, 2p-2\}$, so $v \in \{0, p\}$.\\
 
\textbf{(3)} The statement follows by applying $\pi_{p'}$ to $\iota_p(x) + \iota_p(y) - \iota_p(x+y) \in \{0, p\}$, noting that $\pi_{p'}$ is an additive homomorphism and recalling that $\operatorname{mod}_{p, p'} = \pi_{p'} \circ \iota_p$.\\  

\textbf{(4)} From the definition, we have 
$$\operatorname{mod}_{p_2, p_1} \circ \operatorname{mod}_{p_3, p_2}(x) -\operatorname{mod}_{p_3, p_1}(x) = \pi_{p_1}(\iota_{p_2}(\pi_{p_2}(\iota_{p_3}(x)))) - \pi_{p_1}(\iota_{p_3}(x)) = \pi_{p_1}(\iota_{p_2}(\pi_{p_2}(\iota_{p_3}(x))) - \iota_{p_3}(x)).$$
Write $v = \iota_{p_2}(\pi_{p_2}(\iota_{p_3}(x))) - \iota_{p_3}(x)$. Using the previous work, we know that $p_2|v$, $v \geq -(p_3 - 1)$ and $v \leq 0$, since $\iota_{p_3}(x) \geq 0$. So $v \in \{-tp_2, -(t-1)p_2, \dots, 0\}$, and the claim follows after applying $\pi_{p_1}$.\end{proof}

\begin{proof}[Proof of Proposition~\ref{basicdiffmoduliadd}] Observe immediately that if $\lambda_0 = 0$, we can ensure that $\lambda_1 \alpha(x) + \mu_1 x = 0$ and $\lambda_2 \beta(y) + \mu_2 y = 0$, proving the claim. Therefore, we may assume $\lambda_0 \not= 0$, w.l.o.g. $\lambda_0 = 1$. If $\mu_1 = \mu_2 = 0$ holds, then the function becomes $f \colon (x,y) \mapsto \alpha(x) \beta(y) + \lambda_1 \alpha(x) + \lambda_2 \beta(y)$, which can be made zero, by choosing zero maps for $\alpha$ and $\beta$. If exactly one of $\mu_1, \mu_2$ vanishes, $\mu_1 = 0$ say, then we can pick $\beta$ to ensure that $\lambda_2 \beta(y) + \mu_2 y = 0$, and set $\alpha(x) = 0$ to get $f = 0$.
From now on, assume that $\lambda_1, \lambda_2, \mu_1, \mu_2 \not= 0$.\\
Set $\alpha_1(x) = 0$ and $\beta_2(y) = 0$. This makes $\alpha(x) \beta(y) = 0$ for all choices of $x, y$.  It remains to pick $\alpha_2(x), \beta_1(y)$ so that $(\mu_1 x_1 + \lambda_2 \beta_1(y) + \mu_2 y_1, \lambda_1 \alpha_2(x) + \mu_1 x_2 + \mu_2 y_2)$ takes a small number of values.\\
Set $\beta_1(y) = -\lambda_2^{-1}(\mu_1\operatorname{mod}_{q,p}(y_2) + \mu_2 y_1)$ and $\alpha_2(x) = -\lambda_1^{-1}(\mu_2\operatorname{mod}_{p,q}(x_1) + \mu_1 x_2)$. Hence $f$ becomes
$$f(x,y) = (\mu_1(x_1 - \operatorname{mod}_{q,p}(y_2)), \mu_2(y_2 - \operatorname{mod}_{p,q}(x_1))).$$
Let $\Phi\colon \mathbb{Z}_p \oplus \mathbb{Z}_q \to \mathbb{Z}$ be given by $\Phi(u,v) = \iota_p (\mu_1^{-1} u) + \iota_q(\mu_2^{-1} v)$, noting that $\mu_1, \mu_2 \not=0$. Then,
$$\Phi (f(x,y)) = \iota_p(x_1 - \operatorname{mod}_{q,p}(y_2)) + \iota_q(y_2 - \operatorname{mod}_{p,q}(x_1)).$$
Fixing the set $S = \{-p, 0, p\} + \{-q,0,q\}$, from Lemma~\ref{modProps} we have 
$$\iota_p(x_1 - \operatorname{mod}_{q,p}(y_2)) + \iota_q(y_2 - \operatorname{mod}_{p,q}(x_1)) \in \iota_p(x_1) - \iota_p(\operatorname{mod}_{q,p}(y_2)) + \iota_q(y_2) - \iota_q(\operatorname{mod}_{p,q}(x_1)) + S$$
or, under our notation introduced earlier,
$$\Phi (f(x,y)) \overset{O(1)}{=} \iota_p(x_1) - \iota_p(\operatorname{mod}_{q,p}(y_2)) + \iota_q(y_2) - \iota_q(\operatorname{mod}_{p,q}(x_1)) = \iota_p(x_1) - \iota_q(\pi_q(\iota_p(x_1))) + \iota_q(y_2) - \iota_p(\pi_p(\iota_q(y_2)))$$
Lemma~\ref{modProps} also implies that $\iota_p (\pi_p(v)) \overset{O(\frac{q}{p})}{=} v$ and $\iota_q (\pi_q(v)) \overset{O(1)}{=} v$, when $|v| = O(q)$, from which we conclude that 
$$\Phi (f(x,y)) \overset{O(\frac{q}{p})}{=} \iota_p(x_1)- \iota_p(x_1) + \iota_q(y_2) - \iota_q(y_2) = 0,$$
so the image of the function $f$ is a subset of a preimage of $\Phi$ of a set of size $O(1)$. Fibers of $\Phi$ are of size at most $p$, so the claim follows.\end{proof}



Applying the Proposition~\ref{basicdiffmoduliadd} finishes the proof of the Theorem~\ref{1square}.\end{proof}

\subsection{Using affine maps in the case of two variables}

In this subsection, we further discuss some quadratic expressions involving two variables. A natural map we can try is an affine map $x \mapsto ax +b$, for constants $a,b$. However, if we look at expression $\alpha(x)\beta(y) + \alpha(x) + x + \beta(y) + y$, which was among the ones necessary to discuss in the proof of Theorem~\ref{1square}, it is easy to see that choosing affine maps from $\mathbb{Z}_q$ to $\mathbb{Z}_q$ for $\alpha$ and $\beta$ yields full image, for every $q$. Here we ask ourselves the question when we can use such maps to get a small image of the function defined by the expression.\\

As we shall see later in the paper, in the construction of $A$ with small $2A^2 + kA$, one of the expressions we shall consider has quadratic part of the form $\alpha_1(x_1) \alpha_2(x_2) + (\alpha_1(x_1) +c_1x_1)(\alpha_2(x_2) + c_2 x_2)$, with $c_1, c_2 \not=0$. It turns out that in this case the affine maps can be used as desired maps. We discuss these maps before the construction of $A$ with small $2A^2 + kA$, so that we can focus better on the new ideas needed for that case.

\begin{lemma}(Affine maps solution.)\label{affineMapsLemma} Let $\nu_1, \nu_2 \not= 0$ and $\lambda_1, \lambda_2, \mu_1, \mu_2$ be integers. Then, for any prime $p$ greater than absolute values of all the given integers, we can find affine maps $\alpha, \beta\colon \mathbb{Z}_p \to \mathbb{Z}_p$ such that
$$\alpha(x) \beta(y) + (\alpha(x) +\nu_1x)(\beta(y) + \nu_2 y) + \lambda_1 \alpha(x) + \mu_1 x + \lambda_2 \beta(y) + \mu_2 y$$
is constant.\end{lemma} 

\begin{proof} Let $\alpha(x) := a x + b$ and $\beta(y) := c y + d$, with $a,b,c,d$ to be determined. With this choice of maps, the expression above becomes
$$ (ac + (a+\nu_1)(c + \nu_2)) x y + (2ad + d \nu_1 + \lambda_1 a + \mu_1) x + (2bc + b\nu_2 + \lambda_2 c + \mu_2) y + (2bd + \lambda_1 b + \lambda_2 d).$$
Hence, we need to make sure that 
$$2ac + \nu_2 a + \nu_1 c + \nu_1 \nu_2 = 0,$$
$$2ad + \nu_1 d + \lambda_1 a + \mu_1 = 0,$$
and
$$2bc + \nu_2 b + \lambda_2 c + \mu_2 = 0.$$
This is equivalent to
$$ b = -(\lambda_2 c + \mu_2) /(2c + \nu_2),$$
$$ a = -(\nu_1 c + \nu_1 \nu_2) / (2c + \nu_2)$$
and 
$$ d = \left(\mu_1 (2c + \nu_2) - \lambda_1 \nu_1 (c + \nu_2)\right)/\left(\nu_1 \nu_2\right).$$
Hence, we can pick $a,b,c,d$ so that affine maps make our expression equal to constant iff $\nu_1, \nu_2$ are non-zero.
\end{proof}

\section{Sets $A$ with small $2A^2 + kA$} 

This section is devoted to the proof of the case $l=2$ of the Theorem~\ref{mainResult}.

\begin{theorem}\label{2squares} For any $k \in \mathbb{N}_0$ and any $\epsilon > 0$, there is a natural number $q$, which is a product of distinct, arbitrarily large primes, and a set $A \subset \mathbb{Z}_q$ such that $A-A = \mathbb{Z}_q$, while $|2A^2 + kA| < \epsilon q$. \end{theorem}

\begin{proof} The approach here is similar to the one in the proof of the Theorem~\ref{1square}, however the expressions that arise in this case are more complicated and require new ideas. Once again, the proof is based on the Proposition~\ref{RuzsaArg}. As before, we split all expressions in their quadratic and linear parts, and we may assume that if a variable appears at all in an expression, it must appear in the quadratic part. Next, we consider all the possible cases for the quadratic part, and explain how to make the image of the expression small in each case separately. They are listed sorted by the support size and then by structure. We also have the freedom of renaming the variables. Again, we change the notation slightly; instead of $x_1, x_2, x_3, x_4$ and $\alpha_1, \alpha_2, \alpha_3, \alpha_4$ we use $x,y,z,w$ and $\alpha, \beta, \gamma, \delta$ respectively. The possible cases, w.l.o.g. are (all the $c_i$ are in $\{0,1\}$)  
\begin{enumerate}
\item Support of size 1.
	\begin{enumerate}
	\item The non-linear part must look like $(\alpha(x) + c_1 x)(\alpha(x) + c_2 x) + (\alpha(x) + c_3 x)(\alpha(x) + c_4 x)$.
	\end{enumerate}
\item Support of size 2. We have a few possibilities here.
	\begin{enumerate}
	\item $(\alpha(x) + c_1 x)(\alpha(x) + c_2 x) + (\alpha(x) + c_3 x)(\beta(y) + c_4 y)$
	\item $(\alpha(x) + c_1 x)(\beta(y) + c_2 y) + (\alpha(x) + c_3 x)(\beta(y) + c_4 y)$
	\item $(\alpha(x) + c_1 x)(\alpha(x) + c_2 x) + (\beta(y) + c_3 y)(\beta(y) + c_4 y)$
	\end{enumerate}
\item Support of size 3. We have a few possibilities here.
	\begin{enumerate}
	\item $(\alpha(x) + c_1 x)(\alpha(x) + c_2 x) + (\beta(y) + c_3 y)(\gamma(z) + c_4 z)$
	\item $(\alpha(x) + c_1 x)(\beta(y) + c_2 y) + (\alpha(x) + c_3 x)(\gamma(z) + c_4 z)$
	\end{enumerate}
\item Support of size 4.
	\begin{enumerate}
	\item The non-linear part must look like $(\alpha(x) + c_1 x)(\beta(y) + c_2 y) + (\gamma(z) + c_3 z)(\delta(w) + c_4 w)$.
	\end{enumerate}
\end{enumerate} 
We discuss each of these case separately. However, we use a different order than stated above and deal with easier cases first.\\

\noindent\textbf{Case 1(a).} This is immediate from Corollary~\ref{singleVarExpression}.\\

\noindent\textbf{Case 2(b).} If $c_1 = c_3$ or $c_2 = c_4$, modifying $\alpha(x)$ by adding a suitable multiple $\lambda x$ to it, and modyfing $\beta(y)$ accordingly, we may assume that the quadratic expression is exactly $2\alpha(x) \beta(y)$, which we have already done in Proposition~\ref{basicdiffmoduliadd} (notice that the condition on coefficients in that proposition is satisfied). Hence, w.l.o.g. $c_1 \not= c_3$ and $c_2 \not=c_4$. Then, (after a suitable modification of $\alpha_i$ by affine maps to make $c_1 = c_2 = 0$, $c_3, c_4 \not= 0$), we can apply the Lemma~\ref{affineMapsLemma}, to finish the proof in this case.\\

\noindent\textbf{Case 2(c).} The whole expression in this case is of the form $f_1(x) + f_2(y)$, where $f_1$ is a polynomial of degree at most 2 in $x$ and $\alpha(x)$ and $f_2$ is a polynomial of degree at most 2 in $y$ and $\beta(y)$. Note that we cannot use our arguments about single variable expressions here, as we would only get two sets $S_1, S_2 \subset \mathbb{Z}_q$ of size $o(q)$ such that $f_i$ always takes values in $S_i$, so we would only know that the whole expression takes values in $S_1 + S_2$ which could easily be the whole set of residues. Instead, we recall that the polynomials always attain a small value. This is the content of the next lemma, which is a well-known consequence of Weyl's inequality on exponential sums. Similar results appear in~\cite{Gowers}, we include a proof for completness.    

\begin{lemma}\label{smallPolyVal} Let $d$ be fixed. Then there is an absolute constant $C_d$ such that the following holds. Let $p$ be a prime, and let $a_d, a_{d-1}, \dots, a_0 \in \mathbb{Z}_p$ be given, with $a_d$ non-zero. Then the polynomial $a_dx^d + \dots + a_1 x + a_0$ attains a value in $\{-C_d p^{1-2^{-d}}, \dots, C_d p^{1-2^{-d}}\}$. \end{lemma}

Write $e_p(t)$ for the function $\exp(2\pi i t/p)$. The proof uses discrete Fourier transforms of functions $f \colon \mathbb{Z}_p \to \mathbb{C}$, which we define as $\hat{f}\colon \mathbb{Z}_p \to \mathbb{C}$ with $\hat{f}(r) = \sum_{x \in \mathbb{Z}_p} f(x) e_p(-rx)$. We refer readers to~\cite{Gowers} for more details.

\begin{proof} Write $f(x)$ for the polynomial $a_dx^d + \dots + a_1 x + a_0$. We begin by stating (a special case of) Weyl's inequality.
\begin{theorem}(Weyl's inequality.~\cite{Vaughan}) For every $\epsilon > 0$, and $d \in \mathbb{N}$, there is a constant $C_{\epsilon, d}$ such that for all primes $p$
$$\left|\sum_{x \in \mathbb{Z}_p} e_p(g(x))\right| \leq C_{\epsilon, d} p^{1 + \epsilon - 2^{1-d}}$$
holds for every polynomial $g \in \mathbb{Z}_p[X]$ of degree $d$.
\end{theorem}

Write $F(x)$ for the number of times the polynomial $f$ attains the value $x$. Hence, by Weyl's inequality, there is a constant $C$, independent of $p$ such that $|\hat{F} (r)| \leq C p^{1- 2^{-d}}$ for $r\not=0$, and $\hat{F}(0) = p$. Let $I$ be interval $\{-k, -k + 1, \dots, k\}$. Suppose that $f$ attains no value in $\{-2k, -2k + 1, \dots, 2k\}$. We have
$$\sum_x F(x) I \ast I(x) = 0.$$
Applying Parseval's formula and noting that $\hat{I}(r) \in \mathbb{R}$, we get that 
$$0 = \sum_r \hat{F}(r) \hat{I}(r)^2 = \hat{F}(0) \hat{I}(0)^2 + \sum_{r\not=0} \hat{F}(r) \hat{I}(r)^2 = p (2k+1)^2 + \sum_{r\not=0} \hat{F}(r) \hat{I}(r)^2.$$
Thus,
$$p (2k+1)^2 \leq \sum_{r\not=0} |\hat{F}(r)| \hat{I}(r)^2 \leq \left(\max_{r \not=0} |\hat{F}(r)|\right) \sum_s \hat{I}(s)^2 \leq C p^{1- 2^{-d}}p(2k+1).$$
From this we conclude that $2k+1 \leq C p^{1- 2^{-d}}$, as desired.  
 \end{proof}

Write $N$ for $C_d p^{1-2^{-d}}$. Now, consider $f_1(x)$ as a polynomial in $\alpha(x)$ for every fixed $x$. The lemma guarantees that we can define $\alpha(x)$ so that $f_1(x) \in \{-N, -N + 1, \dots , N\}$. Similarly, for every $y$, we can pick $\beta(y)$ so that $f_2(y) \in \{-N, -N+1, \dots, N\}$, hence we always have $f_1(x) + f_2(y) \in \{-2N, -2N+1, \dots, 2N\}$, as desired.\\

\noindent\textbf{Case 3(a).} We shall take $q$ of the form $q_1 q_2 q_3$, where $q_1, q_2, q_3$ are coprime, and each is a product of distinct arbitrarily large primes. As always, we identify $\mathbb{Z}_q \cong \mathbb{Z}_{q_1} \oplus \mathbb{Z}_{q_2} \oplus \mathbb{Z}_{q_3}$, and we aim to use the identification of coordinates idea. Thus, we set $\alpha_1(x) := -c_1 x_1, \alpha_2(x) := -c_2 x_2$, so that $(\alpha(x) + c_1 x)(\alpha(x) + c_2 x)$ has second and third coordinates equal to zero. We also set $\beta_1(y) := -c_3 y_1, \beta_3(y) := -c_3 y_3$ and $\gamma_2(z) := -c_4 z_2, \gamma_3(z) := -c_4 z_3$. Note that we still have freedom of choice for $\alpha_3, \beta_2, \gamma_1$. Let the linear part of the expression be $d_1 \alpha(x) + d_2 x + d_3 \beta(y) + d_4y + d_5 \gamma(z) + d_6 z$, where the coefficients $d_i$ have the property that $d_{2i} \not= 0$ implies $d_{2i-1} \not= 0$ (since the linear part comes from $\mathbb{N}$-linear combination of $\alpha(x)$ and $\alpha(x) + x$, etc.). The expression becomes
\begin{equation*}\begin{split}
&((-d_1 c_1 + d_2) x_1 + (-d_3c_3 + d_4)y_1 + d_5 \gamma_1(z) + d_6 z_1,\\
&(-d_1 c_2 + d_2) x_2 + d_3 \beta_2(y) + d_4y_2 + (-c_4 d_5 + d_6) z_2,\\
&(\alpha_3(x) + c_1 x_3) (\alpha_3(x) + c_2 x_3) + d_1 \alpha_3(x) + d_2 x_3 + (-d_3 c_3 + d_4)y_3 + ( -d_5 c_4 + d_6) z_3).\end{split}\end{equation*}

We combine the identification of coordinates idea with the fact that polynomials have relatively dense sets of values in the next proposition.

\begin{proposition}\label{strongIdent}(Strong version of the identification of coordinates) Fix $n, d \in \mathbb{N}$. Then there are constants $\epsilon, C > 0$ such that the following holds.  Let $d_1, d_2, \dots, d_n \in \mathbb{N}$ all be at most $d$. Let $2p_n > p_1 \geq p_2 \geq \dots \geq p_n$ be primes. Write $r = p_1p_2 \dots p_n$. Next, let $f_{i,j}\colon\mathbb{Z}_r \to \mathbb{Z}_{p_j}$ be arbitrary maps for every $1 \leq i,j \leq n$. Let for every $1 \leq i \leq n$, $c_i \in \mathbb{Z}_{p_i}^\times$. Finally, let $g_{i, j}\colon \mathbb{Z}_r \to \mathbb{Z}_{p_i}$ be also arbitrary functions for every $1\leq i \leq n, 1 \leq j \leq d_i-1$. Then, we can find maps $\alpha_i \colon \mathbb{Z}_r \to \mathbb{Z}_{p_i}$ such that the expression
\begin{equation*}\begin{split}
(&f_{1, 1}(x_1) + f_{2,1}(x_2) + \dots + f_{n,1}(x_n) + c_1 \alpha_1(x_1)^{d_1} + g_{1,d_1 - 1} (x_1) \alpha_1(x_1)^{d_1 - 1} + \dots + g_{1,1}(x_1) \alpha_1(x_1),\\
& f_{1, 2}(x_1) + f_{2,2}(x_2) + \dots + f_{n,2}(x_n) + c_2 \alpha_2(x_2)^{d_2} + g_{2,d_2 - 1} (x_2) \alpha_2(x_2)^{d_2 - 1} + \dots + g_{2,1}(x_2) \alpha_2(x_2),\\
& \vdots\\
& f_{1, n}(x_1) + f_{2,n}(x_2) + \dots + f_{n,n}(x_n) + c_n \alpha_n(x_n)^{d_n} + g_{n,d_n - 1} (x_n) \alpha_n(x_n)^{d_n - 1} + \dots + g_{n,1}(x_n) \alpha_n(x_n))
\end{split}\end{equation*} 
takes at most $C p_n^{-\epsilon} p_1 p_2 \dots p_n$ values as $x_1, x_2, \dots, x_{n-1}$ and $x_n$ range over all values in $\mathbb{Z}_r$. 
\end{proposition}

Throughout the paper, we will use the prime number theorem (\cite{Ivic}) without explicitly mentioning it. 

\begin{proof} 

Write $q$ for $p_n$ (in fact any prime close to $p_1, p_2, \dots, p_n$ would work). The main idea is to pick $\alpha_1, \dots, \alpha_n$ so that every value $(v_1, v_2, \dots, v_n)$ attained by the expression satisfies $\sum_{i=1}^n\mod_{p_i, q} (v_i) \in S$, for a small subset $S \subset \mathbb{Z}_q$. Partitioning $\mathbb{Z}_{p_1} \oplus \mathbb{Z}_{p_2} \oplus \dots \oplus \mathbb{Z}_{p_n}$ into cosets of $\{0\} \times \dots \times \{0\} \times \mathbb{Z}_{p_n}$, we see the set of values of the expression can take only at most $|S|$ values on each coset, and thus a small number of values in total.\\

We use the Lemma~\ref{smallPolyVal} in order to define $\alpha_i$. Recall that the lemma gives $C', \epsilon > 0$ such that every non-constant polynomial of degree at most $d$ in $\mathbb{Z}_{p_i}$ for any $i$, takes a value in $\{0, 1, \dots, C' q^{1- \epsilon}\}$ (modify the constant coefficient if necessary). For every $i$, we define $\alpha_i$ as follows. We apply the lemma for every fixed $x_i \in \mathbb{Z}_{p_1} \oplus \mathbb{Z}_{p_2} \oplus \dots \oplus \mathbb{Z}_{p_n}$ to the polynomial
$$ c_i t^{d_i} + \sum_{j = 1} ^{d_i - 1} g_{i, j}(x_i) t^j + \sum_{j = 1} ^ n \operatorname{mod}_{p_j, p_i} (f_{i,j} (x_i)).$$
Hence, we can pick $t$, such that this expression takes value in $\{0, 1, \dots, C' q^{1- \epsilon}\} \subset \mathbb{Z}_{p_i}$. We set $\alpha_i(x_i) := t$. Therefore, we have defined $\alpha_i\colon \mathbb{Z}_{p_1} \oplus \mathbb{Z}_{p_2} \oplus \dots \oplus \mathbb{Z}_{p_n} \to \mathbb{Z}_{p_i}$, so that 
$$\operatorname{mod}_{p_i, q} \left(c_i \alpha_i(x_i)^{d_i} + \sum_{j = 1} ^{d_i - 1} g_{i, j}(x_i) \alpha_i(x_i)^j + \sum_{j = 1} ^ n \operatorname{mod}_{p_j, p_i} (f_{i,j} (x_i))\right) \in S \subset \mathbb{Z}_q,$$
where $S = \operatorname{mod}_{p_i, q} (\{0, 1, \dots, C' q^{1-\epsilon}\}) = \{0, 1, \dots, C' q^{1-\epsilon}\}$. To finish the proof, we apply the Lemma~\ref{modProps}.\\

Note that we have
\begin{equation}\begin{split}
&\sum_{i=1}^n \operatorname{mod}_{p_i,q}\left(\sum_{j=1}^n f_{j, i}(x_j) + c_i \alpha_i(x_i)^{d_i} + \sum_{j = 1}^{d_i - 1} g_{i,j} (x_i) \alpha_i(x_i)^j\right)\\
\overset{O_n(1)}{=} &\sum_{i=1}^n \left(\sum_{j=1}^n \operatorname{mod}_{p_i,q}(f_{j, i}(x_j)) + \operatorname{mod}_{p_i,q}(c_i \alpha_i(x_i)^{d_i}) + \sum_{j = 1}^{d_i - 1} \operatorname{mod}_{p_i,q}(g_{i,j} (x_i) \alpha_i(x_i)^j)\right)\\
\overset{O_n(1)}{=} &\left(\sum_{i=1}^n\sum_{j=1}^n \operatorname{mod}_{p_j,q} \circ \operatorname{mod}_{p_i,p_j}(f_{j, i}(x_j))\right) + \left(\sum_{i=1}^n\left(\operatorname{mod}_{p_i,q}(c_i \alpha_i(x_i)^{d_i}) + \sum_{j = 1}^{d_i - 1} \operatorname{mod}_{p_i,q}(g_{i,j} (x_i) \alpha_i(x_i)^j)\right)\right)\\
= &\left(\sum_{i=1}^n\sum_{j=1}^n \operatorname{mod}_{p_i,q} \circ \operatorname{mod}_{p_j,p_i}(f_{i, j}(x_i))\right) + \left(\sum_{i=1}^n\left(\operatorname{mod}_{p_i,q}(c_i \alpha_i(x_i)^{d_i}) + \sum_{j = 1}^{d_i - 1} \operatorname{mod}_{p_i,q}(g_{i,j} (x_i) \alpha_i(x_i)^j)\right)\right)\\
= &\sum_{i=1}^n \operatorname{mod}_{p_i,q}\left(\operatorname{mod}_{p_j,p_i}(f_{i, j}(x_i)) + c_i \alpha_i(x_i)^{d_i} + \sum_{j = 1}^{d_i - 1}g_{i,j} (x_i) \alpha_i(x_i)^j\right) \in nS
\end{split}\end{equation}
We conclude that values $(v_1, v_2, \dots, v_n)$ attained by the expression with the maps $\alpha_i$ defined as above satisfy 
$$\sum_{i=1}^n\operatorname{mod}_{p_i,q} (v_i) \in nS + T,$$
for a set $T$ of size at most $O_n(1)$. Since $nS = \{0, 1, \dots, nC' q^{1-\epsilon}\} \subset \mathbb{Z}_q$, the expression takes at most $O_{n,d}(p_1 p_2 \dots p_{n-1} p_n^{1-\epsilon})$ values, as desired.\end{proof}

The case 3(a) now follows from a straightforward application of the Proposition~\ref{strongIdent}.\\

We deal with the remaining cases in a similar fashion.\\

\noindent\textbf{Case 2(a).} Let the linear part of the expression be $\lambda_1 \alpha(x) + \mu_1 x + \lambda_2 \beta(y) + \mu_2 y$. We shall take $q = q_1 q_2$, for coprime $q_1$ and $q_2$, with $\mathbb{Z}_q \cong \mathbb{Z}_{q_1} \oplus \mathbb{Z}_{q_2}$. We set $\alpha_1(x) := -c_3 x_1$ and $\beta_2(y):= -c_4 y_2$. It remains to choose $\alpha_2\colon\mathbb{Z}_{q_1} \oplus \mathbb{Z}_{q_2} \to \mathbb{Z}_{q_2}$ and $\beta_1\colon\mathbb{Z}_{q_1} \oplus \mathbb{Z}_{q_2} \to \mathbb{Z}_{q_1}$ so that the expression
\begin{equation*}\begin{split}
(&(c_1 - c_3)(c_2 - c_3) x_1^2 - c_3\lambda_1 x_1 + \mu_1 x_1 + \lambda_2 \beta_1(y) + \mu_2 y_1,\\
&(\alpha_2(x) + c_1 x_2)(\alpha_2(x) + c_2 x_2) + \lambda_1 \alpha_2(x) + \mu_1 x_2 - c_4 \lambda_2 y_2 + \mu_2 y_2)\\
=(&(c_1 - c_3)(c_2 - c_3) x_1^2 + (\mu_1 - c_3 \lambda_1) x_1 + \mu_2 y_1 + \lambda_2 \beta_1(y),\\
  &c_1c_2 x_2^2 + \mu_1 x_2 + (\mu_2 - c_4 \lambda_2)y_2 + \alpha_2(x)^2 + ((c_1 + c_2) x_2 +\lambda_1) \alpha_2(x))\\
\end{split}
\end{equation*}
takes small number of values. But, recalling that $\lambda_2 = 0$ implies $\mu_2 = 0$, this follows directly from the Proposition~\ref{strongIdent}, and we may take $q_1, q_2$ to be prime.\\

\noindent\textbf{Case 3(b).} Let the linear part of the expression be $\lambda_1 \alpha(x) + \mu_1 x + \lambda_2 \beta(y) + \mu_2 y + \lambda_3 \gamma(z) + \mu_3 z$. We shall take $q = q_1 q_2 q_3$, for coprime $q_1, q_2$ and $q_3$, with $\mathbb{Z}_q \cong \mathbb{Z}_{q_1} \oplus \mathbb{Z}_{q_2} \oplus \mathbb{Z}_{q_3}$. We set $\alpha_1(x) := -c_1 x_1, \alpha_2(x) := -c_3 x_2, \beta_2(y) := -c_2y_2, \beta_3(y):= -c_2 y_3, \gamma_1(z) := -c_4 z_1$ and $\gamma_3(z):= -c_4 z_3$. It remains to choose $\alpha_3\colon\mathbb{Z}_{q_1} \oplus \mathbb{Z}_{q_2} \oplus \mathbb{Z}_{q_3} \to \mathbb{Z}_{q_3}, \beta_1\colon \mathbb{Z}_{q_1} \oplus \mathbb{Z}_{q_2} \oplus \mathbb{Z}_{q_3} \to \mathbb{Z}_{q_1}$ and $\gamma_2\colon\mathbb{Z}_{q_1} \oplus \mathbb{Z}_{q_2} \oplus \mathbb{Z}_{q_3}\to \mathbb{Z}_{q_2}$ so that the expression
\begin{equation*}\begin{split}
(&(-c_1\lambda_1 + \mu_1)x_1 + \lambda_2 \beta_1(y) + \mu_2 y_1 + (-c_4 \lambda_3 + \mu_3) z_1,\\
&(-c_3\lambda_1 + \mu_1)x_2 + (-c_2\lambda_2 + \mu_2)y_2 + \lambda_3\gamma_2(z) + \mu_3 z_2,\\
& \lambda_1 \alpha_3(x) + \mu_1 x_3 + (-c_2 \lambda_2 + \mu_2) y_3 + (-c_4 \lambda_3 + \mu_3) z_3)\end{split}\end{equation*}
takes small number of values. Once again, recalling that $\lambda_i = 0$ implies $\mu_i = 0$, this follows directly from the Proposition~\ref{strongIdent}, and we may take $q_1, q_2$ and $q_3$ to be prime.\\

\noindent\textbf{Case 4(a).} Let the linear part of the expression be $\lambda_1 \alpha(x) + \mu_1 x + \lambda_2 \beta(y) + \mu_2 y + \lambda_3 \gamma(z) + \mu_3 z + \lambda_4 \delta(w) + \mu_4w$. We shall take $q = q_1 q_2 q_3 q_4$, for coprime $q_1, q_2, q_3$ and $q_4$, with $\mathbb{Z}_q \cong \mathbb{Z}_{q_1} \oplus \mathbb{Z}_{q_2} \oplus \mathbb{Z}_{q_3} \oplus \mathbb{Z}_{q_4}$. We set
\begin{equation*}\begin{split}
&\beta_1(y):= -c_2 y_1, \gamma_1(z)\colon= -c_3 z_1, \delta_1(w) \colon= -c_4 w_1,\\
&\alpha_2(x) \colon= -c_1 x_2, \gamma_2(z)\colon= -c_3 z_2, \delta_2(w) \colon= -c_4 w_2,\\
&\alpha_3(x) \colon= -c_1 x_3, \beta_3(y):= -c_2 y_3, \delta_3(w) \colon= -c_4 w_3,\\
&\alpha_4(x) \colon= -c_1 x_4, \beta_4(y):= -c_2 y_4, \gamma_4(z)\colon= -c_3 z_4.
\end{split}\end{equation*}
We use the Proposition~\ref{strongIdent} to find $\alpha_1, \beta_2, \gamma_3, \delta_4$ so that the expression
\begin{equation*}\begin{split}
(&\lambda_1 \alpha_1(x) + \mu_1 x_1 + (-c_2 \lambda_2 + \mu_2) y_1 + (-c_3 \lambda_3 + \mu_3) z_1 + (-c_4 \lambda_4 + \mu_4) w_1,\\
&(-c_1 \lambda_1 + \mu_1) x_2 + \lambda_2 \beta_2(y) + \mu_2 y_2 + (-c_3 \lambda_3 + \mu_3) z_2 + (-c_4 \lambda_4 + \mu_4) w_2,\\
&(-c_1 \lambda_1 + \mu_1) x_3 + (-c_2 \lambda_2 + \mu_2) y_3 + \lambda_3 \gamma_3(z) + \mu_3z_3 + (-c_4 \lambda_4 + \mu_4) w_3,\\
&(-c_1 \lambda_1 + \mu_1) x_4 + (-c_2 \lambda_2 + \mu_2) y_4 + (-c_3 \lambda_3 + \mu_3) z_4 + \lambda_4 \delta_4(w)+ \mu_4 w_4)
\end{split}\end{equation*}
takes small number of values. This completes the proof of the Theorem~\ref{2squares}.\end{proof}

\subsection{Further discussion of the identification of coordinates idea}

As we have seen in the proof of the Theorem~\ref{2squares}, the Proposition~\ref{strongIdent} was used in a very similar fashion for several cases of expressions. The goal of this short subsection is to take this approach further and see what expressions can be handled using this idea.\\

We temporarily return to the notation of $x_i$ for the variables and $\alpha_i$ for the maps. The value of $x_i$ at coordinate $c$ is denoted by $x_{i,c}$. Observe that when we use Proposition~\ref{strongIdent}, we have to pick some of the maps $\alpha_{i,c}$ to cancel out the mixed quadratic terms like $\alpha_{1,c}(x_1) (\alpha_{2,c}(x_2) + x_{2,c})$. In the proof of the Theorem~\ref{2squares} in the last few cases, given an expression, we used a different coordinate $c$ for every variable $x_i$, and we picked $\alpha_{j,c}$ for $j \not= i$, so that the mixed quadratic terms dissappear. Our goal now is to put all these ideas together in a single proposition. First, we need to set up some useful definitions.\\

Fix an expression $E$ in variables $x_1, x_2, \dots, x_n$. Define a graph $G_E$ on vertices $\{x_1, x_2, \dots, x_n\}$ by adding an edge $x_i x_j$ for every term of the form $(\alpha_i(x_i) + c x_i) (\alpha_j(x_j) + d x_j)$ with $i\not=j$, with multiple edges allowed (so $x_i x_j$ appears the same number of times the relevant terms occur in $E$).

\begin{proposition}\label{acyclicIdent}(Acyclic version of the identification of the coordinates.) Let $E$ be a quadratic expression such that $G_E$ has no cycles (in particular, no repeated edges). Then there is an absolute constant $\epsilon > 0$ such that the following holds. We can find $q$, a product of distinct, arbitrarily large primes, and maps $\alpha_1, \dots, \alpha_n\colon\mathbb{Z}_q\to\mathbb{Z}_q$ such that $E$ takes at most $O(q^{1-\epsilon})$ values. \end{proposition}

\begin{proof} As promised, we will take $q= q_1 q_2 \dots q_n$, with $q_i$ coprime products of distinct primes, suitably chosen. As always, view $\mathbb{Z}_q$ as the direct sum $\mathbb{Z}_{q_1} \oplus \dots \oplus \mathbb{Z}_{q_n}$. Let $c \in [n]$ be an arbitrary coordinate. We start from $x_c$ and traverse the graph $G_E$. (If $G_E$ is disconnected, pick arbitrary vertices in all other components to start the traversal from. For each such starting vertex $x_i$, $i \not= c$, set $\alpha_{i,c} = 0$.) Since the graph is acyclic, we reach every variable at most once, and we visit every edge. When we move along the edge $x_i x_j$, from $x_i$ to $x_j$, that means that there is a term $(\alpha_i(x_i) + a x_i)(\alpha_j (x_j) + bx_j)$ in the expression, and we set $\alpha_{j,c}(x_j) \colon= -b x_{j,c}$, to make the term vanish. Since this is the first time we reach $x_j$, there are no issues with defining $\alpha_{j,c}$.\\
\indent After this procedure, we have defined $\alpha_{i,j}$ for $i \not= j$, so that for every coordinate $c$, the expression $E_c$ no longer has mixed quadratic terms. We still have the freedom of choosing $\alpha_{c,c}$, so we now may apply the Proposition~\ref{strongIdent} to finish the proof.\end{proof}

As we shall see later, depending on the structure of the graph $G_E$, it is not always possible to choose some of the maps $\alpha_{i,c}$ so that the mixed quadratic terms vanish, so there is no obvious way to make the Proposition~\ref{acyclicIdent} more general.

\section{Sets $A$ with small $3A^2 + kA$}

In this section we prove the final case of the main theorem.

\begin{theorem}\label{3squares} For any $k \in \mathbb{N}_0$ and any $\epsilon > 0$, there is a natural number $q$, which is a product of distinct, arbitrarily large primes, and a set $A \subset \mathbb{Z}_q$ such that $A-A = \mathbb{Z}_q$, while $|3A^2 + kA| < \epsilon q$. \end{theorem}

\begin{proof} We proceed like in the proofs of Theorems~\ref{1square} and~\ref{2squares}, except that the details become once again more complicated and the ideas we developed so far, culminating in Proposition~\ref{acyclicIdent}, do not suffice. As usual, the proof is based on Proposition~\ref{RuzsaArg}. We split all expressions in their quadratic and linear parts, and we may assume that if a variable appears at all in an expression, it must appear in the quadratic part. In the first part of the discussion of the possible expressions, we use the notation $x_i$ for variables and $\alpha_i$ for maps, as there can be upto 6 variables involved. Later, we again switch to $x, y, z$ and $\alpha, \beta, \gamma$ notation.\\

Firstly, by Corollary~\ref{singleVarExpression}, we only need to consider expressions with at least two varaibles. Next, we use the Proposition~\ref{acyclicIdent} to treat the expressions with at least 4 variables. We look at the graph $G_E$. Note that if we have an isolated vertex $x_i$ in $G_E$, since $x_i$ appears in the quadratic part, we must have term of the form $(\alpha_i(x_i) + c_1 x_i) (\alpha_i(x_i) + c_2 x_i)$ in $E$. Hence, the number of isolated vertices $v_{is}$ plus the number of edges $e$ is at most 3, which is the number of quadratic terms in $E$.\\

\noindent\textbf{Expression $E$ with exactly 6 variables.} We look at $G_E$. It is a graph on 6 vertices, with $v_{is} + e \leq 3$. Hence, it is a perfect matching, which is acyclic, so the Proposition~\ref{acyclicIdent} applies.\\

\noindent\textbf{Expression $E$ with exactly 5 variables.} Looking at $G_E$, which is a graph on 5 vertices with $v_{is} + e \leq 3$, we see that at most one vertex can have degree greater than 1. The graph $G_E$ is acyclic, so the Proposition~\ref{acyclicIdent} applies.\\

\noindent\textbf{Expression $E$ with exactly 4 variables.} Once again, we analyse $G_E$. It is a graph on 4 vertices with $v_{is} + e \leq 3$. The only way to get a cycle is if the graph has a double edge $x_1 x_2$ and an edge $x_3 x_4$ (after a suitable renaming of variables). Thus, the quadratic part of $E$ is of the form 
$$(\alpha_1(x_1) + c_1 x_1)(\alpha_2(x_2) + c_2 x_2) + (\alpha_1(x_1) + c'_1 x_1)(\alpha_2(x_2) + c'_2 x_2) + (\alpha_3(x_3) + c_3 x_3)(\alpha_4(x_4) + c_4 x_4),$$
where $c_1, c_2, c'_1, c'_2, c_3, c_4 \in\{0,1\}$. If $c_1 = c'_1$ or $c_2 = c'_2$, we can rewrite the quadratic part as a linear combination of only two quadratic terms, so that the graph $G_E$ becomes a matching, and therefore acyclic. Thus, assume that $c_1 \not= c'_1$ and $c_2 \not=c'_2$. But, using the affine maps solution from the Lemma~\ref{affineMapsLemma} we can cancel all the terms in $E$ that involve $x_1$ and $x_2$. Then, w.l.o.g. $E$ becomes an expression with quadratic term
$$(\alpha_3(x_3) + c_3 x_3)(\alpha_4(x_4) + c_4 x_4)$$ 
which we have already done using the basic version of the identification of coordinates idea in Lemma~\ref{basicdiffmoduliadd}.\\

Hence, we may assume that the expression $E$ has either two or three variables. We treat these cases separately. From now on, we use the notation $x,y,z$ for the variables and $\alpha,\beta,\gamma$ for maps.

\subsection{$E$ has two variables $x$ and $y$}

Observe that if there is at most one mixed quadratic term $(\alpha(x) + c_1 x)( \beta(y) + c_2 y)$ in the quadratic part, then once again Proposition~\ref{acyclicIdent} applies. Hence, we may assume that there are at least two such terms in $E$. Suppose now that there all three quadratic terms are of this form, hence the quadratic part is 
$$(\alpha(x) + c_1 x)( \beta(y) + c_2 y) + (\alpha(x) + c_3 x)( \beta(y) + c_4 y) + (\alpha(x) + c_5 x)( \beta(y) + c_6 y),$$
where $c_1, c_2, \dots, c_6\in\{0,1\}$. This constraint on coefficients is crucial. By pigeonhole principle, there are at least two equal coefficients among $c_1, c_3, c_5$, w.l.o.g. $c_1 = c_3$. The quadratic part of $E$ may be written as
$$(\alpha(x) + c_1 x)(2\beta(y) + (c_2+c_4)y) + (\alpha(x) + c_5 x)( \beta(y) + c_6 y),$$
which we treat using Lemma~\ref{basicdiffmoduliadd} if this factorizes further, or using Lemma~\ref{affineMapsLemma} otherwise.\\

It remains to treat the case when there are exactly two mixed terms, so the quadratic part is w.l.o.g.
$$(\alpha(x) + c_1 x)(\alpha(x) + c_2 x) + (\alpha(x) + c_3 x)( \beta(y) + c_4 y) + (\alpha(x) + c_5 x)(\beta(y) + c_6 y).$$

However, we can no longer use the affine maps to cancel out quadratic terms to modify the expression and then apply the Proposition~\ref{acyclicIdent}. Instead, we have to use a different argument, which unfortunately gives significantly worse bounds.

\begin{lemma}\label{2varsProb} Let $E$ be a quadratic expression with quadratic part of the form
$$n_1 \alpha(x)^2 + \alpha(x) (n_2 x + n_3 \beta(y) + n_4 y) + x (n_5 x + n_6 \beta(y) + n_7 y),$$
with $n_1, n_2, \dots, n_7 \in \mathbb{Z}$ and $n_1, n_3 \not= 0$. Then, for every sufficiently large prime $p$, we can find $\alpha, \beta\colon\mathbb{Z}_p \to \mathbb{Z}_p$ such that the expression does not attain every value in $\mathbb{Z}_p$.
\end{lemma}

Immediately, we have the following corollary.

\begin{corollary} \label{2varsProbFull}Let $E$ be a quadratic expression with quadratic part of the form
$$n_1 \alpha(x)^2 + \alpha(x) (n_2 x + n_3 \beta(y) + n_4 y) + x (n_5 x + n_6 \beta(y) + n_7 y),$$
with $n_1, n_2, \dots, n_7 \in \mathbb{Z}$ and $n_1, n_3 \not= 0$. Let $\epsilon > 0$. Then, there is $q$, product of distinct, arbitrarily large primes, and maps $\alpha, \beta\colon\mathbb{Z}_q \to \mathbb{Z}_q$ such that the expression attains at most $\epsilon q$ values.\end{corollary}

\begin{proof} Let $N$ be the bound in Lemma~\ref{2varsProb} such that for all primes $p > N$ we have $\alpha^{(p)}, \beta^{(p)}\colon \mathbb{Z}_p \to \mathbb{Z}_p$ such that the expression evades one value, i.e. all values are confined to a set $S_p$ of size $p-1$. If we now take $q = p_1 p_2 \dots p_n$, a product of distinct primes greater than $N$, then, once again identifying $\mathbb{Z}_q \cong \mathbb{Z}_{p_1} \oplus \dots \oplus \mathbb{Z}_{p_n}$, and defining $\alpha,\beta\colon\mathbb{Z}_q \to \mathbb{Z}_q$ coordinatewise using $\alpha^{(p_i)}, \beta^{(p_i)}$, we have that the expression in $\mathbb{Z}_q$ attains values in $S_{p_1} \times S_{p_2} \times \dots \times S_{p_n}$. Hence, it takes at most $(p_1 - 1) \dots (p_n - 1)$ values. A standard calculation reveals that for $n$ sufficiently large, the number of values becomes $o(q)$. (The $p$ that appears in the sums and products below ranges over primes only.) Indeed,
\begin{equation*}\begin{split}
\prod_{N < p < M} \frac{p-1}{p} &= \operatorname{exp}\left( \sum_{N < p < M} \log\left(1 - \frac{1}{p}\right) \right) = \operatorname{exp}\left( \sum_{N < p < M}  - \frac{1}{p} + O\left(\frac{1}{p^2}\right) \right)\\
&= O\left( \operatorname{exp}\left(  - \sum_{N < p < M} \frac{1}{p}\right) \right) \to 0
\end{split}\end{equation*}
as $M \to \infty$, since $\sum_p \frac{1}{p} = \infty$.\end{proof}

\begin{proof}[Proof of Lemma~\ref{2varsProb}.] Let $\lambda_1 \alpha(x) + \mu_1 x + \lambda_2 \beta(y) + \mu_2 y$ be the linear part of the expression. We will define $\alpha\colon\mathbb{Z}_p \to \mathbb{Z}_p$ essentially by setting each $\alpha(y)$ uniformly independently at random (for technical reasons, for every $x$ we will forbid one value in $\mathbb{Z}_p$). Our aim is to define $\beta$ accordingly so that the expression evades zero value. Hence, for every $y$, we want to find $\beta(y)$ such that there is no $x$ with
\begin{equation}\label{expression}\beta(y) (n_3 \alpha(x) + n_6x + \lambda_2) + \alpha(x) (n_1 \alpha(x) + n_2 x + n_4 y + \lambda_1) + n_5 x^2  + n_7 xy + \mu_1 x + \mu_2 y = 0.\end{equation}
In other words, provided $n_3 \alpha(x) + n_6 x + \lambda_2 \not= 0$ always, we want a value of $\beta(y)$ such that 
\begin{equation}\label{betaEquation}\beta(y) \not= -\frac{1}{n_3 \alpha(x) + n_6 x + \lambda_2} \left(y(n_4\alpha(x) + n_7x + \mu_2) + \alpha(x) (n_1 \alpha(x) + n_2 x + \lambda_1) + n_5 x^2 + \mu_1 x\right),\end{equation}
for all $x \in \mathbb{Z}_p$. Hence, this becomes the requirement that for every fixed $y$, the set 
$$S_y \colon= \left\{ -\frac{1}{n_3 \alpha(x) + n_6 x + \lambda_2} \left(y(n_4\alpha(x) + n_7x + \mu_2) + \alpha(x) (n_1 \alpha(x) + n_2 x + \lambda_1) + n_5 x^2 + \mu_1 x\right) \colon x \in \mathbb{Z}_p\right\}$$
is not the whole set $\mathbb{Z}_p$. We now define $\alpha\colon\mathbb{Z}_p\to\mathbb{Z}_p$ by setting each $\alpha(x)$ independently to be a uniform random variable on $\mathbb{Z}_p \setminus \{-\frac{n_6 x + \lambda_2}{n_3}\}$ (which is fine, as $n_3 \not= 0$).\\

Let $B_y$ be the event that the set $S_y$ is the whole $\mathbb{Z}_p$, i.e. for every $v$ there is $x$ such that
\begin{equation}\label{vEqn} 0 = v (n_3 \alpha(x) + n_6x + \lambda_2) + \left(y(n_4\alpha(x) + n_7x + \mu_2) + \alpha(x) (n_1 \alpha(x) + n_2 x + \lambda_1) + n_5 x^2  + \mu_1 x\right).\end{equation}
Suppose that $B_y$ occurs. We cannot use the same $x$ for two values of $v$, so by counting, for every $v$, we have exactly one $x = x(v)$ such that (\ref{vEqn}) holds. Suppose that we already know this permutation $x(v) = \pi(v)$. The equation is further equivalent to
$$n_1\alpha(\pi(v))^2 + \alpha(\pi(v)) (n_2 \pi(v) + n_4 y + n_3 v + \lambda_1) + n_5 \pi(v)^2 + n_6 \pi(v) v + n_7 \pi(v) y +  \mu_1 \pi(v) + y\mu_2 + v  \lambda_2 = 0.$$
Hence, for every $v$, we know that $\alpha(\pi(v))$ must take one of the two values depending only on $v$, since $n_1 \not=0$. So, given $\pi$, there are at most $2^p$ choices for $\alpha$. Hence, the probability of $B_y$ is $\mathbb{P}(B_y) \leq p! 2^p / (p-1)^p$. By Stirling's formula,  
$$\mathbb{P}(B_y) = O\left(\sqrt{p}\left(\frac{2}{e}\right)^p\right).$$
By the union bound, the probability $\mathbb{P}(\cup_y B_y) = o(1)$, so there is a choice of $\alpha$ such that for all $y$ we have $S_y \not= \mathbb{Z}_p$. For such $\alpha$, we can define $\beta$ so that the expression does not attain every value, proving the lemma.
\end{proof}

Returning to our main argument, the case when the quadratic part is of the form
$$(\alpha(x) + c_1 x)(\alpha(x) + c_2 x) + (\alpha(x) + c_3 x)( \beta(y) + c_4 y) + (\alpha(x) + c_5 x)(\beta(y) + c_6 y).$$
follows directly from Corollary~\ref{2varsProbFull}, since $n_1 = 1, n_3 = 2$.

\subsection{$E$ has three variables}

Finally, we address the case when the quadratic part of $E$ has exactly three variables. Once again, we only need to consider the situation when $G_E$ has a cycle. We know that $G_E$ is a graph on three vertices, with $v_{is} + e \leq 3$. The only there such graphs that have cycles are $xy, xy$ (a repeated edge and an isolated vertex), $xy, xy, xz$ (a repeated edge and an additional edge) and $xy, yz, zx$ (a cycle of length 3).\\

\noindent\textbf{$G_E$ is a repeated edge.}In this case, the quadratic part of the expression is w.l.o.g.
$$(\alpha(x) + c_1 x)(\beta(y) + c_2y) + (\alpha(x) + c_3 x)(\beta(y) + c_4y) + (\gamma(z) + c_5z)(\gamma(z) + c_6z).$$
If $c_1 = c_3$ or $c_2 = c_4$, we can further factorize the expression and apply the Proposition~\ref{acyclicIdent}, to finish the proof. Thus assume that $c_1 \not= c_3$ and $c_2 \not=c_4$.\\
Let the linear part of the expression be $\lambda_1\alpha(x) + \mu_1x+\lambda_2\beta(y) + \mu_2y + \lambda_3\gamma(z) + \mu_3z$. Fix a prime $p$, and apply Lemma~\ref{affineMapsLemma} to the expression 
$$(\alpha(x) + c_1 x)(\beta(y) + c_2y) + (\alpha(x) + c_3 x)(\beta(y) + c_4y) + \lambda_1\alpha(x) + \mu_1x+\lambda_2\beta(y) + \mu_2y$$
to make it constant. Hence, it remains to pick $\gamma\colon\mathbb{Z}_p \to \mathbb{Z}_p$ so that the expression
$$(\gamma(z) + c_5z)(\gamma(z) + c_6z) + \lambda_3\gamma(z) + \mu_3z$$
attains a small number of values, which we can ensure if we apply Lemma~\ref{smallPolyVal} for each $z$ to the polynomial $\gamma(z)^2 + (c_5z +c_6z + \lambda_3) \gamma(z) + c_5c_6z^2 + \mu_3 z$. Provided $p$ is large enough, $\gamma(z)$ can be chosen so that the value of the polynomial is small. This completes the proof in this case.\\

\noindent\textbf{$G_E$ is a 3-cycle.} In this case, the quadratic part of $E$ has three mixed terms, one for each pair of variables among $x,y,z$. More precisely, it is
$$(\alpha(x) + c_1 x)(\beta(y) + c_2 y) + (\beta(y) + c_3y)(\gamma(z) + c_4z) + (\gamma(z) + c_5z)(\alpha(x) + c_6x),$$
where $c_1, \dots, c_6 \in \{0,1\}$. Let the linear part be 
$$\lambda_1 \alpha(x) + \mu_1 x + \lambda_2 \beta(y) + \mu_2y + \lambda_3 \gamma(z) + \mu_3z.$$

First, assume that no further factorization is possible, i.e. $c_1 \not= c_6, c_2 \not= c_3$ and $c_4 \not= c_5$. 
We set $\alpha(x) = -c_1 x + d_1, \beta(y) = -c_3 y + d_2, \gamma(z) = -c_5z + d_3$, so that the expression becomes
\begin{equation*}\begin{split}
d_1 ((c_2 - c_3)y + d_2) &+ d_2 ((c_4 - c_5) z + d_3) + d_3 ((c_6 - c_1) x + d_1) + (\mu_1 - c_1 \lambda_1) x + (\mu_2 - c_3 \lambda_2) y + (\mu_3 - c_5 \lambda_3) z\\
&+ (\lambda_1 d_1 + \lambda_2 d_2 + \lambda_3 d_3).\end{split}\end{equation*}
Rearranging further,
\begin{equation*}\begin{split}
& x (d_3 (c_6 - c_1) + \mu_1 - c_1 \lambda_1) + y (d_1 (c_2 - c_3) + \mu_2 - c_3 \lambda_2) + z(d_2(c_4 - c_5) + \mu_3 - c_5 \lambda_3)\\
+ & (\lambda_1 d_1 + \lambda_2 d_2 + \lambda_3 d_3 + d_1d_2 + d_2d_3 + d_3d_1).
\end{split}\end{equation*}
Setting $d_1 = \frac{\mu_2 - c_3 \lambda_2}{c_3 - c_2}$, $d_2 = \frac{\mu_3 - c_5 \lambda_3}{c_5-c_4}$ and $d_3 = \frac{\mu_1 - c_1 \lambda_1}{c_1 - c_6}$, the expression becomes constant.\\

Now, suppose that w.l.o.g. $c_1 = c_6$. Assume for now that $(c_3 - c_2)(c_4 - c_5) = 0$, we will address the case when this product does not vanish later. The expression becomes 
$$(\alpha(x) + c_1 x)(\beta(y) + c_2 y + \gamma(z) + c_5z) + (\beta(y) + c_3y)(\gamma(z) + c_4z) + \lambda_1 \alpha(x) + \mu_1 x + \lambda_2 \beta(y) + \mu_2y + \lambda_3 \gamma(z) + \mu_3z.$$
We use the identification of coordinates approach. We will take $q = p_1 p_2 p_3$, where $p_1 < p_2 < p_3 < 2p_1$ are arbitrarily large primes. Identify $\mathbb{Z}_q \cong \mathbb{Z}_{p_1} \oplus \mathbb{Z}_{p_2} \oplus \mathbb{Z}_{p_3}$. Our first step is to set
$$\alpha_1(x) = -c_1 x_1, \beta_1(y) = -c_3 y_1 + 1 - \lambda_3, \alpha_2(x) = -c_1x_2, \gamma_2(z) = -c_4 z_2 + 1 - \lambda_2.$$
This way, the quadratic terms vanish in the first two coordinates, and we still have freedom of choosing $\beta_2, \gamma_1$ to cancel the linear terms in $y, z$. We want to do the same for $\alpha_3$, so we set $\beta_3(y) = -c_2 y_3 + 1 - \lambda_1, \gamma_3(z) = -c_5 z_3$. However, with such a choice, the third coordinate of the expression is
\begin{equation*}\begin{split}
&(1-\lambda_1)(\alpha_3(x) + c_1x_3) + ((c_3 - c_2) y_3 + 1 - \lambda_1)((c_4 - c_5)z_3) + \lambda_1 \alpha_3(x) + \mu_1 x_3 + y_3(\mu_2-\lambda_2c_2) + z_3(\mu_3 - \lambda_3 c_5)\\
+ &\lambda_2(1-\lambda_1)\\
=&\alpha_3(x) + ((1-\lambda_1) c_1 + \mu_1) x_3 + (c_3 - c_2)(c_4 - c_5)y_3z_3 + (\mu_2-\lambda_2c_2) y_3 + (\mu_3 - \lambda_3 c_5+ (1 - \lambda_1)(c_4 - c_5)) z_3 + \lambda_2(1-\lambda_1). 
\end{split}\end{equation*} 

Since $(c_3 - c_2)(c_4 - c_5) = 0$, the expression becomes
\begin{equation*} \begin{split}
( & (\mu_1 - c_1\lambda_1)x_1 + (\mu_2 - c_3\lambda_2) y_1 + \gamma_1(z) + (\mu_3 + (1-\lambda_3)c_4)z_1 + \lambda_2(1-\lambda_3)\\
& (\mu_1 - c_1\lambda_1)x_2 + \beta_2(y) + (\mu_2 + c_3(1-\lambda_2))y_2 + (\mu_3 - c_4\lambda_3)z_2 + \lambda_3(1-\lambda_2)\\
& \alpha_3(x) + ((1-\lambda_1) c_1 + \mu_1) x_3 + (\mu_2-\lambda_2c_2) y_3 + (\mu_3 - \lambda_3 c_5+ 1 - \lambda_1) z_3 + \lambda_2(1-\lambda_1)).
\end{split}\end{equation*}

We may now apply the identification of coordfinates idea, using Proposition~\ref{strongIdent}, to finish the proof in this case.\\

Now assume that $(c_3 - c_2)(c_4 - c_5) \not= 0$. We shall take $q = p_1p_2p_3p_4p_5$ and use the additional fourth and fifth coordinates to cancel out the $y_3z_3$ term. Also, using the prime number theorem, we can find arbitrarily large primes such that $p_1 < \dots < p_5 < p_1 + O(\log p_i)$. In the work below it will be essential that all the primes are close in value (although it will not be important to have them this close). Writing $E$ also for the resulting map defined by $\alpha, \beta, \gamma$ and the expression, our aim is to show that 
$$\sum_{i=1}^5\operatorname{mod}_{p_i,p_3} (E_i)$$
takes few values in $\mathbb{Z}_{p_3}$.\\
We use the same choices of $\alpha_1, \alpha_2, \beta_1, \beta_3, \gamma_2, \gamma_3$ as in the case when $(c_3 - c_2)(c_4 - c_5) = 0$. Next, we set $\alpha_4(x) = -c_1 x_4, \beta_4(y) = -\operatorname{mod}_{p_3, p_4} (y_3) - c_3 y_4, \gamma_4(y) = \operatorname{mod}_{p_3, p_4} (z_3) - c_4z_4$. Observe that \begin{equation*}\begin{split}
&\operatorname{mod}_{p_4,p_3} \left((\beta_4(y) + c_3 y_4)(\gamma_4(y) + c_4z_4)\right) + \operatorname{mod}_{p_3,p_3}(y_3 z_3) = \\
&\operatorname{mod}_{p_4,p_3} (-\operatorname{mod}_{p_3, p_4} (y_3)\operatorname{mod}_{p_3, p_4} (z_3)) + y_3z_3 =\\
&\pi_{p_3} \circ \iota_{p_4} ( - \pi_{p_4} \circ \iota_{p_3} (y_3)\pi_{p_4} \circ \iota_{p_3} (z_3)) + y_3z_3\end{split}\end{equation*}

Let $\overline{y_3} = \iota_{p_3} (y_3)$ and $\overline{z_3} = \iota_{p_3}(z_3)$. Hence $\overline{y_3},\overline{z_3} \in \{0,1, \dots, p_3 - 1\}$ are integers such that $\pi_{p_3}(\overline{y_3}) = y_3$ and $\pi_{p_3}(\overline{z_3}) = z_3$ hold. We also have $\iota_{p_4} ( - \pi_{p_4} \circ \iota_{p_3} (y_3)\pi_{p_4} \circ \iota_{p_3} (z_3)) = \iota_{p_4} ( - \pi_{p_4} (\overline{y_3}) \pi_{p_4} (\overline{z_3})) = \iota_{p_4} ( \pi_{p_4} ( -\overline{y_3} \hspace{1pt} \overline{z_3})) $. But the $\iota_{p_4} ( \pi_{p_4} ( -\overline{y_3}\hspace{1pt}\overline{z_3}))$ is an integer $w \in \{0,1, \dots, p_4-1\}$ such that $\pi_{p_4}(w) = \pi_{p_4} ( -\overline{y_3}\hspace{1pt}\overline{z_3})$, thus $w = -\overline{y_3}\hspace{1pt}\overline{z_3} + p_4 t$, for $t = \lceil \frac{\overline{y_3}\hspace{1pt}\overline{z_3}}{p_4} \rceil$. Therefore, with this choice of $t$ we have
\begin{equation*}\begin{split}
&\operatorname{mod}_{p_4,p_3} ((\beta_4(y) + c_3 y_4)(\gamma_4(y) + c_4z_4)) + \operatorname{mod}_{p_3,p_3}(y_3 z_3) = \\
&\pi_{p_3} \circ \iota_{p_4} ( - \pi_{p_4} \circ \iota_{p_3} (y_3)\pi_{p_4} \circ \iota_{p_3} (z_3)) + y_3z_3 = \\
&\pi_{p_3} (-\overline{y_3}\hspace{1pt}\overline{z_3} + p_4 t) + \pi_{p_3} (\overline{y_3})\pi_{p_3} (\overline{z_3})=\\
&\pi_{p_3} (p_4 t) = \pi_{p_3}((p_4 - p_3) t)
\end{split}\end{equation*}

Proceeding further, we use the fifth coordinate to approximate $(p_4 - p_3) t$. To this end, write $M = \lfloor\sqrt{p_4}\rfloor$, $\overline{y_3} = u M + u', \overline{z_3} = vM + v'$, where $u', v' \in \{0, 1, \dots, M-1\}$, $u,v = O(M)$. Observe that $uv$ is a good approximation to $t$
\begin{equation*}\begin{split}
&\left| t - uv\right| = \left| \lceil \frac{\overline{y_3} \hspace{1pt}\overline{z_3}}{p_4} \rceil - uv\right| \leq 1 + \left| \frac{\overline{y_3} \hspace{1pt}\overline{z_3} - p_4uv}{p_4} \right| = 1 + \left| \frac{(uM + u')(vM + v') - p_4uv}{p_4} \right| \\
\leq & 1 + \left| \frac{u'vM + uv' M + u'v'}{p_4} \right| + \left| \frac{uv(p_4- M^2)}{p_4}\right| \leq C_1\sqrt{p_4}
\end{split}\end{equation*}
for some absolute constant $C_1$, since $u,v,u',v', M, |p_4 - M^2| = O(\sqrt{p_4})$. Therefore, we set $\alpha_5 = -c_1 x_5, \beta_5(y) = -\pi_{p_5}(u) - c_3 y_5, \gamma_5(z) = \pi_{p_5}(v (p_4 - p_3)) - c_4z_5$. Note that $\beta_5, \gamma_5$ are well defined, as $u$ depends on $y$ only, and $v$ depends on $z$ only. With $\beta_5$ and $\gamma_5$ so defined we have
\begin{equation*}\begin{split}
&\operatorname{mod}_{p_5,p_3}((\beta_5(y) + c_3y_5)(\gamma_5(z) + c_4z_5)) + \pi_{p_3}(t(p_4-p_3)) =\\
&\pi_{p_3}(\iota_{p_5}(-\pi_{p_5}(u)\pi_{p_5}(v (p_4 - p_3))) + t(p_4 - p_3)) =\\
&\pi_{p_3}(\iota_{p_5}(\pi_{p_5}(-uv(p_4-p_3))) + t(p_4 - p_3)) 
\end{split}\end{equation*}
We also have that $\iota_{p_5}(\pi_{p_5}(-uv(p_4-p_3)))$ is an integer $s \in \{0, 1, \dots, p_5-1\}$ such that $\pi_{p_5}(s) = \pi_{p_5}(-uv(p_4-p_3))$, thus $s = -uv(p_4-p_3) + p_5 t'$, where $t'  = \lceil \frac{uv(p_4-p_3)}{p_5} \rceil \leq C_2\log p_3$, for an absolute constant $C_2$. Therefore, 
\begin{equation*}\begin{split}
&\operatorname{mod}_{p_5,p_3}((\beta_5(y) + c_3y_5)(\gamma_5(z) + c_4z_5)) + \pi_{p_3}(t(p_4-p_3)) =\\
&\pi_{p_3}(\iota_{p_5}(\pi_{p_5}(-uv(p_4-p_3))) + t(p_4 - p_3)) =\\
&\pi_{p_3}(-uv(p_4-p_3) + p_5 t' + t(p_4-p_3)) =\\
&\pi_{p_3}((t-uv)(p_4-p_3) + p_5 t')
\end{split}\end{equation*}

Summing up the work done so far we conclude that
\begin{equation*}\begin{split}
\operatorname{mod}_{p_3,p_3}(&y_3 z_3) + \operatorname{mod}_{p_4,p_3} (\beta_4(y) + c_3 y_4)(\gamma_4(y) + c_4z_4)) + \operatorname{mod}_{p_5,p_3}((\beta_5(y) + c_3y_5)(\gamma_5(z) + c_4z_5))\\
=& y_3 z_3 + \operatorname{mod}_{p_4,p_3}(-\operatorname{mod}_{p_3,p_4}(y_3)\operatorname{mod}_{p_3,p_4}(z_3)) + \operatorname{mod}_{p_5,p_3}(\pi_{p_5}(-uv(p_4 - p_3))) \in S_1,
\end{split}\end{equation*}
where $S_1 \subset\mathbb{Z}_{p_3}$ is the set defined by $\{\pi_{p_3}(a(p_4 - p_3) + p_5 b)\colon a,b \in \mathbb{Z}, |a| \leq C_1 \sqrt{p_4}, |b| \leq C_2 \log p_3\}$. In particular $|S_1| = O(\sqrt{p_3} \log^2p_3)$. Finally, we put everything together, using the Lemma~\ref{modProps}. Recall the definitions (the maps $\beta_4, \gamma_4$ and $\gamma_5$ below are slightly modified to cancel the term $(c_3-c_2)(c_4-c_5)y_3z_3$ instead of just $y_3 z_3$) 
\begin{equation*}\begin{split}
\alpha_1(x) &= -c_1 x_1, \beta_1(y) = -c_3 y_1 + 1 - \lambda_3,\\
\alpha_2(x) &= -c_1x_2, \gamma_2(z) = -c_4 z_2 + 1 - \lambda_2,\\
\beta_3(y) &= -c_2 y_3 + 1 - \lambda_1, \gamma_3(z) = -c_5 z_3,\\
\alpha_4(x) &= -c_1 x_4, \beta_4(y) = -(c_3-c_2)\operatorname{mod}_{p_3, p_4} (y_3) - c_3 y_4, \gamma_4(y) = (c_4-c_5)\operatorname{mod}_{p_3, p_4} (z_3) - c_4z_4,\\
\alpha_5 &= -c_1 x_5, \beta_5(y) = -\pi_{p_5}(u) - c_3 y_5, \gamma_5(z) = \pi_{p_5}(v (p_4 - p_3)(c_3-c_2)(c_4-c_5)) - c_4z_5.
\end{split}\end{equation*} 
Thus,
\begin{equation*}\begin{split}
\sum_{i=1}^5& \operatorname{mod}_{p_i, p_3} (E_i) = \sum_{i=1}^5  \operatorname{mod}_{p_i, p_3}((\alpha_i(x) + c_1 x_i)(\beta_i(y) + c_2 y_i + \gamma_i(z) + c_5z_i) + (\beta_i(y) + c_3y_i)(\gamma_i(z) + c_4z_i)\\
&+\lambda_1 \alpha_i(x) + \mu_1 x_i + \lambda_2 \beta_i(y) + \mu_2y_i + \lambda_3 \gamma_i(z) + \mu_3z_i)\\
=& \operatorname{mod}_{p_1, p_3}(\gamma_1(z) + (\mu_1 - c_1\lambda_1)x_1 + (\mu_2 - c_3 \lambda_2)y_1 + (\mu_3 + c_4(1-\lambda_3))z_1 + \lambda_2(1-\lambda_3))\\
&+\operatorname{mod}_{p_2, p_3}(\beta_2(y) + (\mu_1 - c_1\lambda_1) x_2 + (\mu_2 + c_3(1-\lambda_2))y_2 + (\mu_3 - c_4\lambda_3)z_2 + \lambda_3(1-\lambda_2))\\
&+\alpha_3(x)+ y_3z_3(c_3-c_2)(c_4-c_5) + x_3(c_1(1-\lambda_1) + \mu_1)  + (\mu_2 - \lambda_2 c_2)y_3 + ((1-\lambda_1)(c_4-c_5) - c_5\lambda_3 + \mu_3)z_3\\
&+\lambda_2(1-\lambda_1)\\
&+\operatorname{mod}_{p_4,p_3}( -(c_3-c_2)(c_4-c_5)\operatorname{mod}_{p_3,p_4}(y_3)\operatorname{mod}_{p_3,p_4}(z_3) + (\mu_1 - \lambda_1c_1)x_4 - (c_3-c_2)\lambda_2 \operatorname{mod}_{p_3,p_4}(y_3)\\
&+ (\mu_2 - c_3 \lambda_2)y_4 + (c_4-c_5)\lambda_3\operatorname{mod}_{p_3,p_4}(z_3) + (\mu_3 - \lambda_3 c_4)z_4)\\
&+\operatorname{mod}_{p_5,p_3}(-\pi_{p_5}(u) \pi_{p_5}(v(p_4-p_3)(c_3-c_2)(c_4-c_5)) + (\mu_1 - \lambda_1c_1) x_5 - \lambda_2 \pi_{p_5}(u) + (\mu_2 - \lambda_2 c_3)y_5 \\
&+\lambda_3 \pi_{p_5}(v(p_4-p_3)(c_3-c_2)(c_4-c_5)) + (\mu_5 - \lambda_3 c_4)z_5)\\
\overset{O(1)}{=}&\pi_{p_3}(\iota_{p_3}(\alpha_3(x)) + (\mu_1 - c_1\lambda_1)\iota_{p_1}(x_1) + (\mu_1 - c_1\lambda_1) \iota_{p_2}(x_2) + (c_1(1-\lambda_1) + \mu_1)\iota_{p_3}(x_3) + (\mu_1 - \lambda_1c_1)\iota_{p_4}(x_4)\\
&+ (\mu_1 - \lambda_1c_1) \iota_{p_5}(x_5)\\
&+ \iota_{p_2}(\beta_2(y)) + (\mu_2 - c_3 \lambda_2)\iota_{p_1}(y_1) + (\mu_2 + c_3(1-\lambda_2))\iota_{p_2}(y_2) + (\mu_2 - \lambda_2 c_2)\iota_{p_3}(y_3) - (c_3-c_2)\lambda_2 \iota_{p_4}(\operatorname{mod}_{p_3,p_4}(y_3))\\
&+ (\mu_2 - c_3 \lambda_2)\iota_{p_4}(y_4) - \lambda_2 \iota_{p_5}(\pi_{p_5}(u)) + (\mu_2 - \lambda_2 c_3)\iota_{p_5}(y_5)\\
&+ \iota_{p_1}(\gamma_1(z)) + (\mu_3 + c_4(1-\lambda_3))\iota_{p_1}(z_1) + (\mu_3 - c_4\lambda_3)\iota_{p_2}(z_2) + ((1-\lambda_1)(c_4-c_5) - c_5\lambda_3 + \mu_3) \iota_{p_3}(z_3)\\
 &+ (c_4-c_5)\lambda_3\iota_{p_4}(\operatorname{mod}_{p_3,p_4}(z_3))+ (\mu_3 - \lambda_3 c_4)\iota_{p_4}(z_4) + \lambda_3 \iota_{p_5}(\pi_{p_5}(v(p_4-p_3)(c_3-c_2)(c_4-c_5))) + (\mu_5 - \lambda_3 c_4)\iota_{p_5}(z_5))\\
&+(c_3-c_2)(c_4-c_5) (y_3 z_3 - \operatorname{mod}_{p_4,p_3}(\operatorname{mod}_{p_3,p_4}(y_3)\operatorname{mod}_{p_3,p_4}(z_3)) - \operatorname{mod}_{p_5,p_3}(\pi_{p_5}(uv(p_4 - p_3)))).
\end{split}\end{equation*}
Finally, we set $\alpha_3, \beta_2, \gamma_1$ to cancel the linear $x,y,z$ terms respectively:
\begin{equation*}\begin{split}
\alpha_3(x) &= -\pi_{p_3}((\mu_1 - c_1\lambda_1)\iota_{p_1}(x_1) + (\mu_1 - c_1\lambda_1) \iota_{p_2}(x_2) + (c_1(1-\lambda_1) + \mu_1)\iota_{p_3}(x_3) + (\mu_1 - \lambda_1c_1)\iota_{p_4}(x_4)\\
&+ (\mu_1 - \lambda_1c_1) \iota_{p_5}(x_5))\\
\beta_2(y) &= - \pi_{p_2}((\mu_2 - c_3 \lambda_2)\iota_{p_1}(y_1) + (\mu_2 + c_3(1-\lambda_2))\iota_{p_2}(y_2) + (\mu_2 - \lambda_2 c_2)\iota_{p_3}(y_3) - (c_3-c_2)\lambda_2 \iota_{p_4}(\operatorname{mod}_{p_3,p_4}(y_3)) \\
&+ (\mu_2 - c_3 \lambda_2)\iota_{p_4}(y_4) - \lambda_2 \iota_{p_5}(\pi_{p_5}(u)) + (\mu_2 - \lambda_2 c_3)\iota_{p_5}(y_5)))\\
\gamma_1(z) &= -\pi_{p_1}((\mu_3 + c_4(1-\lambda_3))\iota_{p_1}(z_1) + (\mu_3 - c_4\lambda_3)\iota_{p_2}(z_2) + ((1-\lambda_1)(c_4-c_5) - c_5\lambda_3 + \mu_3) \iota_{p_3}(z_3)\\
&+ (c_4-c_5)\lambda_3\iota_{p_4}(\operatorname{mod}_{p_3,p_4}(z_3))+ (\mu_3 - \lambda_3 c_4)\iota_{p_4}(z_4) + \lambda_3 \iota_{p_5}(\pi_{p_5}(v(p_4-p_3)(c_3-c_2)(c_4-c_5)))\\
&+ (\mu_5 - \lambda_3 c_4)\iota_{p_5}(z_5)))
\end{split}\end{equation*}

With this choice of $\alpha,\beta,\gamma$ we have
$$\sum_{i = 1}^5 \operatorname{mod}_{p_i, p_3}(E_i) \overset{O(1)}{=} (c_3-c_2)(c_4-c_5) (y_3 z_3 - \operatorname{mod}_{p_4,p_3}(\operatorname{mod}_{p_3,p_4}(y_3)\operatorname{mod}_{p_3,p_4}(z_3)) - \operatorname{mod}_{p_5,p_3}(\pi_{p_5}(uv(p_4 - p_3))))$$
which takes small number of values.\\

\noindent\textbf{$G_E$ is has a repeated edge and another single edge.} In this case, the quadratic part of the expression is w.l.o.g.
$$(\alpha(x) + c_1 x)(\beta(y) + c_2y) + (\alpha(x) + c_3 x)(\beta(y) + c_4y) + (\alpha(x) + c_5x)(\gamma(z) + c_6z).$$
If $c_1 = c_3$ or $c_2 = c_4$, we can further factorize the expression and apply the Proposition~\ref{acyclicIdent}, to finish the proof. Thus assume that $c_1 \not= c_3$ and $c_2 \not=c_4$. Since all $c_i \in \{0,1\}$, we must have $c_5 \in \{c_1, c_3\}$, so w.l.o.g. $c_5 = c_1$.\\

We now discuss a limitation of the usual approach based on the identification of coordinates idea. Basically, we always try to cancel out the quadratic terms by taking some of the $\alpha_i, \beta_i, \gamma_i$ to be affine, while we use the rest to cancel out the linear terms in $x_i, y_i, z_i$. Let us try the same strategy here. Temporarily we work in $\mathbb{Z}_p \oplus \mathbb{Z}_p \oplus \dots \oplus \mathbb{Z}_p$ to ignore the difficulties that arise from moving from one modulus to another one. For technical reasons, we use a slightly unusual indexing of $n+2$ coordinates by $-1,0, \dots, n$. Start by using the coordinate -1 to get a free $\gamma_{-1}$ which is later used to cancel the linear terms involving $z$. Thus, we set $\alpha_{-1}(x) = -c_1 x_{-1}$ and $\beta_{-1}(y) = -c_4y_{-1}$. Similarly, try to use the coordinate 0 to get a free $\beta_0$ map. Rewriting the expression as
$$\beta(y) (2\alpha(x) + (c_1 + c_3)x) + y((c_2 + c_4)\alpha(x) + (c_1c_2 + c_3c_4)x) + (\alpha(x) + c_5x)(\gamma(z) + c_6z),$$
we see that we need to set $\alpha_0(x) = -\frac{c_1 + c_3}{2} x_0 + C$, for a constant $C$ and $\gamma_0(z) = -c_6z_0$. The issue is that we get a term $x_0y_0$ with a non-zero coefficient. The natural thing to do now is to try to cancel somehow this term. During this digression, we forget about the linear terms (in any case, we can cancel them by remaining free $\alpha_i, \beta_i, \gamma_i$).\\

The most natural thing is to set $\gamma_i(z) = -c_6 z_i$ for $i=1,2,\dots, n$ (as further mixed quadratic terms involving $z$ seem even harder to cancel). Hence, the question is whether we can find linear maps $\alpha_1, \dots, \alpha_n, \beta_1, \dots, \beta_n$, each a linear combination of $x_0, x_1, \dots, x_n$ or $y_0, y_1, \dots, y_n$ such that (w.l.o.g. $c_1 = c_2 = 0$ and $c_3 = c_4 = 1$)
\begin{equation}\label{impos1}\sum_{i=1}^n \alpha_i(x) \beta_i(x) + (\alpha_i(x) + x_i)(\beta_i(y) + y_i) = 0.\end{equation} 
Write $\alpha_i(x) = \sum_{j=0}^n A_{ij} x_j$ and $\beta_i(y) = \sum_{j = 0}^n B_{ij} y_j$. Let $\delta_{ij}$ equal 1 if $i = j$ and zero otherwise. Expanding the (\ref{impos1}) we obtain
\begin{equation}\begin{split}\label{impos2}
\sum_{i=1}^n \left( \left(\sum_{j=0}^n A_{ij} x_j\right)\left(\sum_{k=0}^n B_{ik}y_k\right) + \left(\sum_{j=0}^n (A_{ij} + \delta_{ij}) x_j\right)\left(\sum_{k=0}^n (B_{ik} + \delta_{ik})y_k\right)\right)=\\
\sum_{j=0}^n\sum_{k=0}^n \left( \sum_{i=1}^n 2A_{ij}B_{ik} + A_{ij} \delta_{ik} + \delta_{ij} B_{ik} + \delta_{ij} \delta_{ik} \right) x_jy_k.
 \end{split}\end{equation}  
Hence, we require that for every $j,k \in \{0, 1, \dots, n\}$, which are not both zero, we have $\sum_{i=1}^n 2A_{ij}B_{ik} + A_{ij} \delta_{ik} + \delta_{ij} B_{ik} + \delta_{ij} \delta_{ik} = 0$, while for $j=k=0$ this expression is non-zero (to cancel the initial $x_0y_0$ term). We now define two $(n+1) \times (n+1)$ matrices $P,Q$, with entries indexed by $\{0,1\dots, n\} \times \{0,1, \dots, n\}$, by setting $P_{ji} = A_{ij}$ when $i \geq 1$ and $P_{j0} = 0$, and $Q_{ik} = B_{ik}$ if $i \geq 1$ and $Q_{0k} = 0$. Let $I'$ be the matrix of all zeros except $I'_{ii} = 1$ for $i \geq 1$, and let $J$ be the matrix consisting of zeros only, except $J_{00} = 1$. We rewrite (\ref{impos2}) as a matrix equation
$$2PQ + P I' + Q I' + I' = \lambda J$$
for some non-zero $\lambda$. However, this is the same as
$$(2P + I')(2Q + I') = 2 \lambda J - I'.$$ 
But comparing ranks we have
$$\operatorname{rank}(2 \lambda J - I') = \operatorname{rank}((2P + I')(2Q + I')) \leq \operatorname{rank}(2P + I') \leq n < n + 1 = \operatorname{rank}(2 \lambda J - I')$$
which is a contradiction. Hence, this case requires a different approach.\\

Finally, we construct the desired maps for this expression. By adding linear terms to $\alpha, \beta,\gamma$, we may assume that the expression is 
\begin{equation}\label{finalExp}
\alpha(x) \beta(y) + (\alpha(x) + c_1x)(\beta(y) + c_2y) + \alpha(x)\gamma(z) + \lambda_1\alpha(x) + \mu_1 x + \lambda_2\beta(y) + \mu_2 y + \lambda_3 \gamma(z) + \mu_3 z
\end{equation}
for some coefficients $c_1, c_2 \in \{-1,1\}, \lambda_1, \lambda_2, \lambda_3 \in \mathbb{N}_0, \mu_1, \mu_2, \mu_3 \in \mathbb{Z}$. Let us begin by observing that in most cases there is a rather simple solution, which strangely we could not generalize to work for all choices of coefficients. Try setting $\alpha(x) = A, \beta(y) = -c_2y + B$, for some constants $A, B$ and suppose we work in $\mathbb{Z}_q$, where $q$ is a product of distinct, arbitrarily large primes (so that all the coefficients and related expressions are coprime with $q$). With these choices, the expression~(\ref{finalExp}) becomes
\begin{equation*}\begin{split}
&A(-c_2y + B) + (A + c_1 x)B + A \gamma(z) + \lambda_1 A + \mu_1 x + \lambda_2(-c_2y + B) + \mu_2 y + \lambda_3 \gamma(z) + \mu_3 z = \\
&y(-c_2A - c_2\lambda_2 + \mu_2) + x(c_1 B + \mu_1) + \gamma(z)(A + \lambda_3) + \mu_3z + (2AB + \lambda_1 A +\lambda_2 B).
\end{split}\end{equation*}
Further, set $B = -\mu_1 c_1$, (recall that $c_1, c_2 \in \{-1, 1\}$ so $c_1^{-1} = c_1, c_2^{-1} = c_2$) so that the coefficient of $x$ above vanishes. We try to pick $A$ such that coefficient of $y$ also becomes zero, setting $A = c_2 \mu_2 - \lambda_2$. If $A + \lambda_3 \not= 0$, then we can pick $\gamma_3$ to cancel the $z$ term, and the expression actually becomes constant. Otherwise, assume that $c_2 \mu_2 - \lambda_2 + \lambda_3 = 0$. In this case, we prove the following proposition, and the full result is then a consequence of a simple number-theoretic calculation.

\begin{proposition}\label{finalPQ} Let $c_1, c_2, \lambda_1, \lambda_2, \lambda_3, \mu_1, \mu_2, \mu_3 \in \mathbb{Z}$ be some fixed coefficients, such that $c_1, c_2 \in \{-1, 1\}$ and $c_2 \mu_2 - \lambda_2 + \lambda_3 - c_2 \not= 0$. Then, for all sufficiently large primes $p, q$, obeying $q < p < 2q$, we may find maps $\alpha, \beta, \gamma\colon\mathbb{Z}_{pq} \to \mathbb{Z}_{pq}$ such that the expression~(\ref{finalExp}) misses at least $p-q$ values.
\end{proposition} 

\begin{proof} As always, $\mathbb{Z}_{pq}$ is viewed as $\mathbb{Z}_p \oplus \mathbb{Z}_q$. In the first coordinate, we set $\alpha_1(x) =  c_2 \mu_2 - \lambda_2 - c_2, \beta_1(y) = -c_2y -\mu_1c_1, \gamma_1(z) = \frac{-\mu_3 z + \delta_1(z) + D}{c_2\mu_2 + \lambda_3 - \lambda_2 - c_2}$, with $\delta_1(z)$ to be chosen and a constant $D$. After a suitable choice of $D$, the first coordinate of the expression becomes
$y_1 - \delta_1(z)$.\\
On the other hand, we shall use the second coordinate to evade some of the values. To this end, we generalize the Lemma~\ref{2varsProb}, with a similar proof. 

\begin{lemma}\label{probLemma2} Let $S$ be a set, and $q$ a prime. Let $f\colon S\to \mathbb{Z}_q$ be any map, and let $c_1, c_2, \lambda_1, \lambda_2, \mu_1, \mu_2 \in \mathbb{Z}$ be any coefficients. Then, provided $|S| q^2 \cdot q! < (q-1)^q$ we may pick $\alpha, \beta_s\colon \mathbb{Z}_q \to \mathbb{Z}_q$ for all $s \in S$, such that 
\begin{equation}\alpha(x) \beta_s(y) + (\alpha(x) + c_1 x)(\beta_s(y) + c_2 y) + \lambda_1 \alpha(x) + \mu_1x  + \lambda_2 \beta_s(y)  +\mu_2 y + f(s)\label{evasionEqn}\end{equation}
never takes value zero.
\end{lemma}

\begin{proof}[Proof of Lemma~\ref{probLemma2}.] We proceed similarly as in the proof of Lemma~\ref{2varsProb}, starting by defining each $\alpha(x)$ independently, uniformly at random in $\mathbb{Z}_q \setminus \{-2^{-1}(c_1x + \lambda_2)\}$, with this single value omitted for technical reasons.\\
For each $y$ and $s \in S$, we want to pick $\beta_s(y)$, so that~(\ref{evasionEqn}) does not vanish for any $x$. Let $E_{y,s}$ be the event that we cannot do this, i.e. that, having fixed $y,s$ for every value $\beta$, we can find $x$ such that  
\begin{equation}\label{prePerm}\alpha(x) \beta + (\alpha(x) + c_1 x)(\beta + c_2 y) + \lambda_1 \alpha(x) + \mu_1x  + \lambda_2 \beta + \mu_2 y + f(s).\end{equation}
If $E_{y,s}$ occurs, observe that (\ref{prePerm}) cannot hold for distinct $\beta_1, \beta_2$ with the same choice of $x$, since this equation can be rewritten as
$$\beta(2\alpha(x) + c_1 x + \lambda_2) + y(c_2 \alpha(x) + c_1 c_2 x + \mu_2) + \lambda_1 \alpha(x) + \mu_1 x + f(s)$$
and by the choice of $\alpha$, the coefficient of $\beta$ is never zero. Hence, if $\pi\colon \mathbb{Z}_q \to \mathbb{Z}_q$ is the map that sends each $\beta$ to the corresponding value of $x$ for which the (\ref{prePerm}) vanishes, we must have $\pi$ injective, which is thus a bijection.\\
Suppose furthermore that we know $\pi$ as well. Note that in this case we can almost determine $\alpha$. Indeed, for all $\beta$ we have
\begin{equation*}\begin{split}
0=& \beta(2\alpha(\pi(\beta)) + c_1 \pi(\beta) + \lambda_2) + y(c_2 \alpha(\pi(\beta)) + c_1 c_2 \pi(\beta) + \mu_2) + \lambda_1 \alpha(\pi(\beta)) + \mu_1 \pi(\beta) + f(s)\\
 =& \alpha(\pi(\beta))(2\beta + yc_2 + \lambda_1) + \beta(c_1 \pi(\beta) + \lambda_2) + y(c_1 c_2 \pi(\beta) + \mu_2) +  \mu_1 \pi(\beta) + f(s)
\end{split}\end{equation*}
Substituting $\beta = \pi^{-1}(\beta')$, we obtain
$$\alpha(\beta')(2\pi^{-1}(\beta') + yc_2 + \lambda_1) + \pi^{-1}(\beta')(c_1 \beta' + \lambda_2) + y(c_1 c_2 \beta' + \mu_2) +  \mu_1 \beta' + f(s) = 0$$ 
for all $\beta' \in \mathbb{Z}_q$, so $\alpha(\beta')$ is uniquely determined for all $\beta'$ such that $2\pi^{-1}(\beta') + yc_2 + \lambda_1 \not=0$, i.e. for $q-1$ values. So there are at most $q$ ways to pick $\alpha$, and in conclusion, the probability of $E_{y,s}$ is $\mathbb{P}(E_{y,s}) \leq q \cdot q!/(q-1)^q$. Finally, we have
$$\mathbb{P}(\cup_{y,s} E_{y,s}) \leq \sum_{y,s} \mathbb{P}(E_{y,s}) \leq |S| q^2 \frac{q!}{(q-1)^q} < 1,$$
so it is possible to choose $\alpha$ for which all other maps can be defined so that (\ref{evasionEqn}) never vanishes.
\end{proof}

Set $\gamma_2 = 0$. Let $\overline{y_1} = \iota_p(y_1), t = \iota_q(\mu_3 z_2) \in \mathbb{Z}$. We define $\delta_1(z) = \pi_p(t)$, so the first coordinate becomes $\pi_p(\overline{y_1} - t)$. We set $f\colon \mathbb{Z}_p \to \mathbb{Z}_q$, by $f(y_1) = \pi_q(\overline{y_1})$. Apply Lemma~\ref{probLemma2} to the $\mathbb{Z}_q$, $S = \mathbb{Z}_p$, and the expression 
$$\alpha_2(x_2) \beta_{2, y_1}(y_2) + (\alpha_2(x_2) + c_1x_2)(\beta_{2,y_1}(y_2) + c_2y_2) + \lambda_1\alpha_2(x_2) + \mu_1 x_2 + \lambda_2\beta_{2,y_1}(y_2) + \mu_2 y_2 + f(y_1)$$  
to define $\alpha_2, \beta_{2,y_1}\colon \mathbb{Z}_q \to \mathbb{Z}_q$ to make it non-zero always. Note that we may apply the lemma since $p q^2 q! < (q-1)^q$, whenever $q < p < 2q$, for sufficiently large $q$. We define $\beta_2(y)$ as $\beta_{2,y_1}(y_2)$. Finally, we show that values $(\pi_p(r), -\pi_q(r)) \in \mathbb{Z}_p \oplus \mathbb{Z}_q$ are not attained for $r \in \{0,1, \dots, p-q-1\}$.\\

Suppose that $r \in \{0,1, \dots, p-q-1\}$ and suppose that the expression takes value $(\pi_p(r), -\pi_q(r))$. Thus, the first coordinate gives $\pi_p(\overline{y_1} - t) = \pi_p(r)$, so $p$ divides $\overline{y_1} - t - r$, so either $\overline{y_1} \leq t + r - p$, $\overline{y_1} = t+r$, or $\overline{y_1} \geq t + r + p$. But, $\overline{y_1} \in \{0,1, \dots, p-1\}, t \in \{0,1, \dots, q-1\}$ and $r \in \{0,1, \dots, p-q - 1\}$, so we must have $\overline{y_1} = t+r$.\\

Next, let $v$ stand for the value of 
$$\alpha_2(x_2) \beta_{2, y_1}(y_2) + (\alpha_2(x_2) + c_1x_2)(\beta_{2,y_1}(y_2) + c_2y_2) + \lambda_1\alpha_2(x_2) + \mu_1 x_2 + \lambda_2\beta_{2,y_1}(y_2) + \mu_2 y_2.$$ 
By the definition of $\alpha_2, \beta_{2, y_1}$, we always have $v + f(y_1) \not = 0$.   If the second coordinate equals $-\pi_q(r)$, then we have $0 = v + \mu_3 z_2 + \pi_q(r) = v + \pi_q(t) + \pi_q(r) = v + \pi_q(t + r) = v + \pi_q(\overline{y_1}) = v + f(y_1)$, which is impossible. \end{proof}

\begin{corollary} Let $c_1, c_2, \lambda_1, \lambda_2, \lambda_3, \mu_1, \mu_2, \mu_3 \in \mathbb{Z}$ be some fixed coefficients, such that $c_1, c_2 \in \{-1, 1\}$ and $c_2 \mu_2 - \lambda_2 + \lambda_3 - c_2 \not= 0$. Let $\epsilon > 0$ be any small real. Then, we can find $q$, a product of arbitrarily large distinct primes and maps $\alpha, \beta, \gamma\colon \mathbb{Z}_q \to \mathbb{Z}_q$ such that the expression~(\ref{finalExp}) takes at most $\epsilon q$ values in $\mathbb{Z}_q$.\end{corollary}

\begin{proof} We proceed as follows. Look at all the primes $2^k < q_1 < q_2 < \dots < q_m < (1 + \frac{1}{3})2^k$ and $(1 + \frac{2}{3})2^k < p_1 < p_2 < \dots < p_n < 2^{k+1}$. For $k$ sufficiently large, by the prime number theorem, $n, m \geq \Omega(2^k/k)$. For $k$ sufficiently large, pairs of primes $p_i, q_i$ satisfy the conditions of Proposition~\ref{finalPQ}, which we apply to obtain $\alpha_i, \beta_i, \gamma_i \colon \mathbb{Z}_{p_i q_i} \to \mathbb{Z}_{p_i q_i}$ so that the expression~(\ref{finalExp}) misses at least $p_i - q_i$ values in $\mathbb{Z}_{p_i q_i}$. In other words, the expression~(\ref{finalExp}) takes at most $(1-\frac{1}{10p_i})p_iq_i$ values in $\mathbb{Z}_{p_iq_i}$. Let $P_k = \{p_1, p_2, \dots, p_{\min\{m,n\}}\}$, and let $Q_k$ be the product of all $p_iq_i$. Viewing $\mathbb{Z}_{Q_k}$ as a direct sum of $\mathbb{Z}_{p_iq_i}$, we can therefore define $\alpha, \beta, \gamma\colon \mathbb{Z}_{Q_k} \to \mathbb{Z}_{Q_k}$ coordinatewise using $\alpha_i, \beta_i, \gamma_i$, so that the expression~(\ref{finalExp}) attains at most $\prod_{p \in P_k}(1-\frac{1}{10p})Q_k \leq \exp(-\frac{c}{k}) Q_k$ values in $\mathbb{Z}_{Q_k}$, for some positive constant $c$.\\ 
Finally, taking $\mathbb{Z}_{Q_k} \oplus \mathbb{Z}_{Q_{k+1}} \oplus \dots \oplus \mathbb{Z}_{Q_N}$, and using the maps $\alpha,\beta,\gamma$ on each $\mathbb{Z}_{Q_i}$ separately, makes the expression~(\ref{finalExp}) take at most $\prod_{i = k}^N \exp(-\frac{c}{i}) = \exp(-c \sum_{i = k}^n \frac{1}{i})$ proportion of values in $\mathbb{Z}_{Q_k} \oplus \mathbb{Z}_{Q_{k+1}} \oplus \dots \oplus \mathbb{Z}_{Q_N}$, which goes to zero as $N$ goes to infinity, as desired.\end{proof} 
This finishes the proof of the Theorem~\ref{3squares}.\end{proof}

\section{Concluding remarks}

We conclude the paper with some problems and several questions related to the intgredients used in our construction. Firstly, the main question here is still the following.

\begin{question} Suppose that $A \subset \mathbb{Z}_q$ has $A-A = \mathbb{Z}_q$ and let $a_k, a_{k-1}, \dots, a_1 \in \mathbb{N}$. How small can $a_k A^k + a_{k-1}A^k + \dots + a_1A$ be? What is the answer when $q$ is square-free/product of $O(1)$ primes/prime? When can we get a power saving, i.e. $|a_k A^k + a_{k-1}A^k + \dots + a_1A| \leq q^{1-\epsilon}$?
\end{question}

The next natural question is about the number of values attained by expressions.

\begin{question} Let $k \in \mathbb{N}$ be given. We consider expressions in variables $x_1, x_2, \dots, x_k$ and maps $\alpha_1(x_1), \alpha_2(x_2), \dots, \alpha_k(x_k)$. Let $E$ be any $\mathbb{N}$-linear combination of products of terms of the form $\alpha_i(x_i)$ or $\alpha_i(x_i) + x_i$. Is there a choice of a $q \in \mathbb{N}$ and maps $\alpha_i \colon \mathbb{Z}_q \to \mathbb{Z}_q$ such that $E$ attains only $o(q)$ values in $\mathbb{Z}_q$? Is there a choice for which we have a power-saving, i.e. $E$ attains only $O(q^{1-\epsilon})$ values? What if $q$ is square-free/product of $O(1)$ primes/prime?\end{question}

We remark that in our construction, there was a power-saving choice for most of the expressions. In fact, the only ones for which our arguments do not lead to a power-saving are
$$\alpha(x)^2 + \alpha(x)\beta(y) + (\alpha(x) + x)(\beta(y) + y) + \lambda_1 \alpha(x) + \mu_1 x + \lambda_2 \beta(y) + \mu_2 y + \lambda_3 \gamma(z) + \mu_3 z$$
and
$$\alpha(x)\gamma(z) + \alpha(x)\beta(y) + (\alpha(x) + x)(\beta(y) + y) + \lambda_1 \alpha(x) + \mu_1 x + \lambda_2 \beta(y) + \mu_2 y + \lambda_3 \gamma(z) + \mu_3 z,$$
(for a specific choice of $\lambda_i, \mu_i$).\\

Returning once again to the identification of coordinates idea, it turns out that Proposition~\ref{basicdiffmoduliadd} is nearly optimal for some expressions, provided $p$ and $q$ are close. Namely, consider expression $E = \alpha'(x) \beta'(y) + (\alpha'(x) + x) + (\beta'(y) + y) + 1$. Putting $\alpha(x) = \alpha'(x) + 1, \beta(y) = \beta'(y) + 1$, the expression becomes $E = \alpha(x) \beta(y) + x + y$. 

\begin{observation}Let $p$ and $q$ be distinct primes. Given any maps $\alpha, \beta\colon \mathbb{Z}_{pq} \to \mathbb{Z}_{pq}$, the expression $\alpha(x) \beta(y) + x + y$ attains at least $\Omega(\min\{p,q\})$ values in $\mathbb{Z}_{pq}$. \end{observation}

\begin{proof} We begin by observing that if $\alpha(x)$ is not invertible for some choice of $x$, viewing $\mathbb{Z}_{pq}$ as $\mathbb{Z}_p \oplus \mathbb{Z}_q$, for some coordinate $c \in \{1,2\}$, we have $E_c = x_c + y_c$. Letting $y_c$ vary, we obtain at least $\min\{p,q\}$ values.\\
\indent Therefore, assume that all $\alpha(x)$ are invertible in $\mathbb{Z}_{pq} \cong \mathbb{Z}_p \oplus \mathbb{Z}_q$. Fix some $x$. Consider all values $v_1, v_2, \dots, v_r$ of $E(x,y)$, (where $E(x,y)$ is evaluation of the expression for the given choice of $x,y$), as $y$ ranges over $\mathbb{Z}_{pq}$. We may assume $r \leq \frac{1}{10}\min\{p,q\}$, otherwise we are done. Hence, we obtain a partition $Y_1 \cup Y_2 \cup \dots \cup Y_r = \mathbb{Z}_{pq}$, where $E(x,y) = v_i$ if $y \in Y_i$. Call a pair $y_1, y_2$ \emph{invertible} if $y_1 - y_2$ is invertible in $\mathbb{Z}_{pq}$. Observe that in each set $Y_i$, there are at least $\max\{|Y_i| (|Y_i| - p - q + 1) / 2, 0\}$ invertible pairs. However, if $E(x,y_1) = E(x,y_2)$ for an invertible pair $y_1,y_2$, then $\alpha(x) \beta(y_1) + y_1 = \alpha(x) \beta(y_2) + y_2$, so $\beta(y_1) - \beta(y_2)$ is invertible, and $\alpha(x) = \frac{y_1 - y_2}{\beta(y_2) - \beta(y_1)}$. Thus, for every invertible pair $y_1, y_2$ there is a value $w(y_1, y_2)$ such that $E(x,y_1) = E(x,y_2)$ implies $\alpha(x) = w(y_1, y_2)$.\\
For a fixed $w$, take $x$ such that $\alpha(x) = w$, and consider the partition $Y_1 \cup \dots \cup Y_r = \mathbb{Z}_{pq}$ as above. Firstly, take $R$ to be the set of indiced $i$ such that $|Y_i| \geq 2(p+q)$. Thus, $\sum_{i \notin R} |Y_i| < r \cdot 2(p+q) \leq \frac{1}{5}\min\{p,q\} (p+q) \leq \frac{2}{5} pq$. Hence, $\sum_{i \in R} |Y_i| > \frac{3}{5}pq$. Therefore, we obtain that the number of invertible pairs $\{y_1, y_2\}$ that have value $w(y_1, y_2) = \alpha(x) = w$ is at least
$$\sum_{i=1}^r \max\{|Y_i|(|Y_i| - p - q + 1)/2, 0\} \geq \sum_{i \in R} |Y_i|(|Y_i| - p - q + 1)/2 \geq \sum_{i \in R} |Y_i| (p+q) /2 \geq \frac{3}{10} pq(p+q).$$
\indent If $\alpha$ attains at most $2(p+q)$ values, we simply consider $E(x,y)$ for fixed $y$. The expression then attains at least $pq/2(p+q)$ values, thus the claim follows, so we may assume that $\alpha$ attains more than $2(p+q)$ values. But then, for every value $w$ of $\alpha$, we have at least $\frac{3}{10} pq(p+q)$ invertible pairs $\{y_1, y_2\}$ with $w(y_1, y_2) = w$, so the total number of invertible pairs is at least $\frac{3}{10} pq(p+q) \cdot 2(p+q) > p^2q^2$, which is a contradiction.\end{proof}

It could be interesting to better understand the minimum image size for this expression. Furthermore, recall that in the case of prime modulus, we only achieved that $E$ is not surjective.

\begin{question} Let $\alpha,\beta\colon \mathbb{Z}_p \to \mathbb{Z}_p$ be maps and $p$ prime. What is the smallest number of values that the expression $\alpha(x) \beta(y) + x + y$ must attain? \end{question}

Finally, we pose the question of improving the bounds in Lemma~\ref{singleVarLemma}.

\begin{question} Suppose that $c_1, c_2, \dots, c_d$ are never simultaneously zero. How large a set $F$ can we take? \end{question}

\end{document}